\documentclass[]{sn-jnl}

\usepackage{graphicx}%
\usepackage{multirow}%
\usepackage{amsmath,amssymb,amsfonts}%
\usepackage{amsthm}%
\usepackage{mathrsfs}%
\usepackage[title]{appendix}%
\usepackage{xcolor}%
\usepackage{textcomp}%
\usepackage{manyfoot}%
\usepackage{booktabs}%
\usepackage{algorithm}%
\usepackage{algorithmicx}%
\usepackage{algpseudocode}%
\usepackage{listings}%
\usepackage[inline]{enumitem}

\usepackage{textcomp}%
\usepackage{multirow}%

\newcommand{\tabstyle}{%
  \footnotesize
  \setlength{\tabcolsep}{3pt}
  \renewcommand{\arraystretch}{1.05}
}

\newcommand{\SmallDot}{\textcolor{blue}{\footnotesize\textbullet}}
\newcommand{\CDot}{\(\cdot\)}

\usepackage{fancyhdr}

\fancypagestyle{firstpage}{%
\fancyhead[L]{\small Accepted for publication in \emph{Advances in Applied Probability}.}

}

\newtheorem{thm}{Theorem}[section]
\newtheorem{prop}{Proposition}[section]%
\newtheorem{rem}{Remark}[section]%
\newtheorem{lem}{Lemma}[section]

\numberwithin{equation}{section}

\newtheorem{defn}{Definition}[section]%

\begin{document}

\title[Original Article]{De Finetti's Control for Refracted Skew Brownian Motion}


\author[1]{\fnm{Zhongqin} \sur{Gao}}\email{zhongqingaox@126.com}

\author[1]{\fnm{Yan} \sur{Lv}}\email{lvyan1998@aliyun.com}

\author*[2]{\fnm{Xiaowen} \sur{Zhou}}\email{xiaowen.zhou@concordia.ca}

\affil[1]{\orgdiv{School of Mathematics and Statistics}, \orgname{Nanjing University of Science and Technology}, \orgaddress{\street{Xiaolingwei Street}, \city{Nanjing}, \postcode{210094}, \state{Jiangsu}, \country{China}}}

\affil*[2]{\orgdiv{Department of Mathematics and Statistics}, \orgname{Concordia University}, \orgaddress{\street{de Maisonneuve Blvd. West}, \city{Montreal}, \postcode{H3G 1M8}, \state{Quebec}, \country{Canada}}}


\abstract{In this paper we propose a refracted skew Brownian motion as a risk model with endogenous regime switching, which generalizes the refracted diffusion risk process introduced by Gerber and Shiu. We consider an optimal dividend problem for the refracted skew Brownian risk model and identify sufficient conditions, respectively, for barrier strategy, band strategy and their variants to be optimal.}





\maketitle
\thispagestyle{firstpage}

\section{Introduction}\label{sec1}
The optimal dividend problem, a cornerstone of stochastic control theory aimed at maximizing expected discounted dividend payments until ruin, has been extensively explored.
For Brownian motion with drift, the optimal dividend problem has been analyzed by \cite{Gerber2004}, yielding barrier optimality.
More generally, \cite{Shreve1984} studied optimal consumption for It\^{o} diffusions with absorbing and reflecting boundaries and a discounted penalty at absorption. \cite{Paulsen2003} analyzed optimal dividend policies for It\^{o} diffusions under solvency constraints with Lipschitz-continuous drift and diffusion coefficients. \cite{Asmussen2000} solved a joint control problem of dividends and excess-of-loss reinsurance for a diffusion process whose drift and volatility depend on the retention level. \cite{Decamps2007} examined the interaction between dividends and the timing of a growth investment in a diffusion model with piecewise-constant drift that switches endogenously at a stopping time. These diffusion-based formulations typically lead to barrier-type (or band-type, in joint control) optimal policies.

For the Cram\'{e}r-Lundberg risk process, \cite{Gerber1969} showed that the optimal dividend strategy is the so-called band strategy by discrete approximation and limiting argument, and for the particular case of exponentially distributed claim amounts, the band strategy collapses to a barrier strategy. This result was rederived by means of viscosity theory in \cite{Azcue2005}. \cite{Albrecher2008} derived the optimal dividend strategy for the Cram\'{e}r-Lundberg risk model with interest which is again of band type and for exponential claim sizes collapses to a barrier strategy.
For a more general risk process, namely the spectrally negative L\'{e}vy process (SNLP) (see \cite{Bertoin1996}), \cite{Avram2007} gave a sufficient condition involving the generator of the L\'{e}vy process for the optimality of the barrier strategy. \cite{Loeffen2008} showed that the optimal strategy is a barrier strategy if the L\'{e}vy measure has a completely monotone density. \cite{Kyprianou2009} further showed that the optimal strategy is a barrier strategy whenever the L\'{e}vy measure has a density which is log-convex.
\cite{Avram2015} identified necessary and sufficient conditions for optimality of single and impulse-band dividend strategies under a general ruin penalty, allowing for the presence or absence of fixed per-dividend transaction costs.
In addition, \cite{Avanzi2009} provided a review of dividend strategies.

Most dividend-control models are interface-free: there are no fixed thresholds that trigger a change in the dynamics (no skew or refraction). Typically, the models encode state dependence by smooth coefficients rather than threshold-driven regime switches.
However, real-world market frictions, cumulative shocks and regulatory interventions generate threshold effects that interface-free models cannot capture. This motivates a threshold-based, state-dependent specification with an interface as the driver of regime changes.

We therefore propose a skew Brownian surplus process with a two-valued drift, called refracted skew Brownian motion, that couples skew transmission at a fixed threshold with refraction. More specifically, the model introduces a fixed threshold (interface) at which the dynamics change: the drift is refracted across the threshold (yielding piecewise drift), and asymmetric transmission (skew) is implemented via a local-time term that assigns unequal up- and down-crossing tendencies.
Such a two-valued drift encodes transitions between low- and high-risk states, offering a regime-switching mechanism. Moreover, we use this framework to characterize optimal payout policies and to quantify how skew and regime-specific drifts influence the resulting optimal strategies.

We first introduce the skew Brownian motion.
Let process $X\equiv(X_t)_{t\geq0}$, defined on a filtered probability space $(\mathcal{F}, (\mathcal{F}_t)_{t\geq 0}, \mathbb{P}) $, be a solution to the following stochastic differential equation (SDE) with a singular drift.
\begin{align}\label{JD_a11}
	X_t=X_0+\int_0^{t} \sigma (X_s) \mathrm{d}B_s+\int_{\mathbb{R}} \nu(\mathrm{d} x) L^X(t, x),\ \ \ \ t\geq 0,
\end{align}
where $\sigma: \mathbb{R}\rightarrow \mathbb{R}$ is a nonnegative bounded measurable function, $B$ denotes a standard one-dimensional Brownian motion with initial value $0$, $\nu(\mathrm{d} x)$ is a bounded measure on $\mathbb{R}$, $L^X(t, x)$ is the symmetric local time at level $x$ up to time $t\in\mathbb{R}_+$ of process $X$, i.e.
\begin{align}\label{loctime}
	L^X(t,x):=\lim_{\varepsilon\to 0+}\frac{1}{2\varepsilon}\mathrm{Lebesgue}(s\leq t: |X_s-x|\leq\varepsilon).
\end{align}

SDEs of type \eqref{JD_a11} have been studied in \cite{Gall1984}, \cite{Engelbert1985} and \cite{Etore2018}, where existence and uniqueness of a strong solution is proved
under conditions such that $\sigma$ is uniformly elliptic, bounded and of finite variation, and $\nu$ has a finite mass with $|\nu(\{x\})|<1$ for any  $x\in\mathbb{R}$. More generally, a value $x$ such that $|\nu(\{x\})|=1$ corresponds to a reflection of the process over or below this point depending on the sign of $\nu(\{x\})$.
For a given real constant $\beta$ and Dirac measure $\delta_0$ at $0$, when $\sigma\equiv1$ and $\nu(\mathrm{d} x)=\beta \delta_0(\mathrm{d}x)$, we obtain
\begin{align}\label{JD_a22}
	X_t=X_0+B_t+\beta L^X(t, 0),\ \ \ \ t\geq 0,
\end{align}
the so called skew Brownian motion  that was initially studied in \cite{Ito1965} and \cite{Walsh1978}. It is shown in \cite{Harrison1981}  that SDE \eqref{JD_a22} has a unique strong solution if and only if $\left| \beta \right|\leq1$ where $\left| \beta \right|=1$ corresponds to the reflected Brownian motion.
We refer to \cite{Lejay2006} and references therein for a review on skew Brownian motion.

Skew Brownian motion, as an extension of standard Brownian motion, introduces asymmetries in its transition probabilities and boundary behaviors, making it an useful tool for modeling systems with inherent directional biases or state-dependent interactions. It also finds many applications in mathematical finance.
\cite{Rossello2012} studied the arbitrage under skew Brownian motion.
\cite{Gairat2017} pointed it out that in  a driftless two-valued local volatility model, the underlying price, after rescaling, follows the dynamic of skew Brownian motion with two-valued drift. They further obtained formulas for pricing of European options  using joint density for the skew Brownian motion.
\cite{Alvarez2017} studied the optimal stopping  problem arising in the timing of an irreversible investment when the underlying follows a skew Brownian motion.
\cite{Hussain2023} considered the pricing of American options and the corresponding optimal stopping  problem with asset price dynamic following the Azzalini Ito-McKean skew Brownian motion, which is a specific case of skew Brownian motions represented as the sum of a standard Brownian motion and an independent reflecting Brownian motion.
\cite{Hussain2024} investigated the probabilistic distribution functions of maximum of skew Brownian motion and stock price process driven by maximum of skew Brownian motion.

Refracted Brownian motion extends the standard Brownian motion by introducing a state-dependent drift that changes at a specified threshold, making it a versatile tool for modeling surplus dynamics in risk theory. This model is particularly suited to scenarios where premium rates or other key parameters adjust based on reserve levels, and was introduced in \cite{Gerber2006} as the so called refracted diffusion risk process.
\cite{Renaud2021} showed that imposing a linear bound on the dividend rate yields refracted-diffusion surplus process and a mean-reverting optimal policy.
The refracted Brownian motion also constitutes a special case of refracted spectrally negative L\'evy processes introduced in \cite{Kyprianou2010}.
	
We now introduce the refracted skew Brownian motion.
For $\sigma\equiv 1$ and
$$\nu(\mathrm{d} x)= \beta \delta_a(\mathrm{d}x)+(\mu_+\mathbf{1}_{\{x> a\}}+\mu_-\mathbf{1}_{\{x< a\}})\mathrm{d} x,$$
where $\mu_+, \mu_-\in \mathbb{R}$, $\mathbf{1}_{\cdot}$ denotes the indicator function and $\delta_a$ denotes the Dirac measure at $a> 0$, the SDE \eqref{JD_a11} becomes
\begin{eqnarray}\label{JD_a}
	X_t=X_0+B_t+\beta L^X(t, a)+\mu_+\int_0^t \mathbf{1}_{\{X_s> a\}}\mathrm{d} s +\mu_-\int_0^t \mathbf{1}_{\{X_s< a\}}\mathrm{d} s,\ \ t\geq 0.
\end{eqnarray}
If $\beta\in(-1,1)$, then SDE \eqref{JD_a} has a unique solution (see \cite{Gall1984}) the refracted skew Brownian motion, that is the focus of this paper.
If $\beta=0$, process $X$ becomes the so called refracted diffusion risk process. We refer  to \cite{Mazzonetto2016} for numerical issues involving refracted skew Brownian motion.

The refracted skew Brownian motion can be used as a toy model of  endogenous regime-switching process  in which the process has distinct dynamics depending on whether it takes value above or below threshold $a$ where the local time term can be interpreted as the cost or reward associated to the switching of  regimes. This simple model already allows the optimality of a  class  of  barrier and band strategies together with their variants, and the fewer model parameters helps to better understand the interplay between the them. More sophisticated and realistic models with state-dependent diffusion coefficient and (or) jumps can also be proposed and studied in the future.

Our approach is a modification of previous work on optimal dividend problem in risk theory literature.
To identify the optimal strategies, we first prove the corresponding Hamilton-Jacobi-Bellman inequalities that characterize the value function.  Since barrier dividend strategies often serve as the optimal strategies for various surplus processes, we then propose different versions of barrier strategies and for each barrier strategy find an explicit expression of the associated value function.
Applying the Hamilton-Jacobi-Bellman inequalities we identify conditions for each of the barrier strategy to be optimal.

In another interesting finding, we show that a band strategy can be optimal in a broad class of settings, whether band or barrier is optimal depends on the interaction among interface skew, regime-specific drifts and other model parameters. Sufficient conditions are identified for different versions of band strategies to be optimal.
The above setup of the model allows us to analyze how the phase transition of different optimal strategies is affected by the skew and different regimes, which addresses key challenges and contributes to the broader understanding of the stochastic systems with endogenous regime switching.

The rest of the paper is arranged as follows. The two-sided exit problem is solved in Section \ref{exit_problem} for the  refracted skew Brownian risk process. The  Hamilton-Jacobi-Bellman inequalities are shown in Section \ref{HJB}. In Sections \ref{barrier_strategy} and \ref{band_strategy} we find conditions for barrier strategy and band strategy to be optimal, respectively. Numerical illustrations are provided in Section \ref{example}.
\section{Solutions to the exit problems}\label{exit_problem}
In this section, we derive  explicit expressions of the Laplace transforms of exit times for the refracted skew Brownian motion, which provides a theoretical basis for the follow-up study. The law and the expectation with respect to $X$ issued at $x\in \mathbb{R}$ are denoted as $\mathbb{P}_x$ and $\mathbb{E}_x$, respectively.
For any $y\in\mathbb{R}$, define the first hitting time for  process $X$ by
\begin{eqnarray*}
&\tau_{y}:=\inf\{t\geq 0, X_t=y\},
\end{eqnarray*}
with the convention $\inf\emptyset=\infty$.
For any $y< x<z$, define the first exit time of the interval $(y,z)$ for the process $X$ by
\begin{eqnarray*}
&\tau_{y, z}:=\tau_y\wedge \tau_{z}=\min\{\tau_{y}, \tau_{z}\}.
\end{eqnarray*}

To obtain the Laplace transforms, we first  find the general solutions $g_{1,q}(x)$ and $g_{2,q}(x)$ in $C(\mathbb{R})\cap C^{2}(\mathbb{R}\backslash \{a\})$  to the following differential equation
\begin{align}\label{g1g2g3}
&\frac{1}{2}g''(x)+\mu_+\mathbf{1}_{\{x> a\}}g'(x)+\mu_-\mathbf{1}_{\{x< a\}}g'(x)=qg(x),
\end{align}
with $(1+\beta) g'(a+)=(1-\beta)g'(a-)$ for $\beta\in (-1, 1)$ and $q\geq 0$. For $q\geq 0$, let
\begin{eqnarray}
&&g_{1,q}(x):=e^{\rho_{1}^+(x-a)} {\mathbf{1}_{\{x> a\}}}+{\Big(c_1(q)e^{\rho_{2}^-{(x-a)}}+\big(1-c_1(q)\big) e^{\rho_{1}^-{(x-a)}} \Big)}{\mathbf{1}_{\{x\leq a\}}},\label{g1}\\
&&g_{2,q}(x):=\Big(\big(1-c_2(q)\big)e^{\rho_{2}^+(x-a)}+c_2(q) e^{\rho_{1}^+(x-a)} \Big){1_{\{x> a\}}}{+e^{\rho_{2}^-(x-a)}}{1_{\{x\leq a\}}},\label{g2}
\end{eqnarray}
where both $\rho_1^-<0$ and $\rho_2^->0$ satisfy
\begin{align}\label{dfgh}
&\frac{1}{2}{\rho_i^-}^2+\mu_-\rho_i^--q=0, i=1,2,
\end{align}
and both $\rho_{1}^+<0$ and $\rho_{2}^+>0$ satisfy
\begin{align}\label{dfgh212}
&\frac{1}{2}{\rho_i^+}^2+\mu_+\rho_i^+-q=0, i=1,2,
\end{align}
i.e.
\begin{align}\label{dfgh21245}
&\rho_{1}^{\pm}:=-(\mu_{\pm}+\sqrt{\mu_{\pm}^2+2q}), \ \ \ \ \ \ \rho_{2}^{\pm}:=-\mu_{\pm}+\sqrt{\mu_{\pm}^2+2q}.
\end{align}
Since
\begin{eqnarray}\label{g1qx}
&&(1+\beta) g_{i,q}'(a+)=(1-\beta)g_{i,q}'(a-),\ \  i=1,2,
\end{eqnarray}
we have
\begin{align}\label{c1c266}
&c_1(q):=\frac{(1+\beta) \rho_{1}^+-(1-\beta)\rho_{1}^-}
{(1-\beta)(\rho_{2}^--\rho_{1}^-)},\ \ \ \ \ \
c_2(q):=\frac{(1+\beta) \rho_{2}^+-(1-\beta)\rho_{2}^-}{(1+\beta)(\rho_{2}^+-\rho_{1}^+)}.
\end{align}
In particular, for $\mu_-=\mu_+=0$, we have $c_1(q)=\frac{-\beta}{1-\beta}$ and $c_2(q)=\frac{\beta}{1+\beta}$.  The range of $c_i(q), i=1, 2$ is provided below, with the proof in Appendix \ref{c11c223}.
\begin{lem}\label{c1c201}
$c_i(q)<1$ for all $q>0$ and $i=1, 2$. Further, we have
$c_i(q)<0$ if and only if $(1+\beta) \rho_{i}^+<(1-\beta)\rho_{i}^-$.
\end{lem}
\begin{thm}\label{lap 1}
For any $q>0$, $x, y, z\in\mathbb{R}$ and $x\in(y, z)$, we have
\begin{eqnarray}
 &&\mathbb{E}_{x}[e^{-q \tau_{z}}; \tau_{z}<\tau_{y}]=
 \frac{w(x,y)}{w(z,y)},\label{e1}\\
 &&\mathbb{E}_{x}[e^{-q \tau_{y}}; \tau_{y}<\tau_{z}]=\frac{w(x,z)}{w(y,z)},\label{e2}
\end{eqnarray}
where $g_{1,q}(x)$ and $g_{2,q}(x)$ are defined in \eqref{g1} and \eqref{g2}, respectively, and
\begin{eqnarray}\label{wxy}
&&w(x,y):=g_{2,q}(x)g_{1,q}(y)-g_{1,q}(x)g_{2,q}(y).
\end{eqnarray}
\end{thm}
\begin{proof}
For $q>0$ and $i=1,2$, applying the It\^{o}-Tanaka-Meyer formula (c.f. \cite{Lejay2006}, Section 6, and \cite{Protter2005}, Theorem 70) to $\big(e^{-q t}g_{i,q}(X_{t})\big)_{t\geq 0}$, we have
\begin{align*}
&e^{-q t} g_{i,q}(X_{t})=g_{i,q}(x) + M_t^{(i)}
+ \big(\frac{1+\beta}{2} g_{i,q}'(a+)-\frac{1-\beta}{2}g_{i,q}'(a-)\big)\int_0^{t}e^{-q s} L^{X}(\mathrm{d}s, a)\\
& \ \ \
+\int_0^{t} e^{-q s}\big(\frac{1}{2}g_{i,q}''(X_{s-})
+ (\mu_+\mathbf{1}_{\{X_{s-}> a\}} +\mu_-\mathbf{1}_{\{X_{s-}< a\}})g_{i,q}'(X_{s-})-q g_{i,q}(X_{s-})\big)\mathrm{d}s,
\end{align*}
where $M_{t}^{(i)} := \int_0^t e^{-q s}g_{i,q}'(X_{s-})\mathrm{d} B_s$, $i = 1, 2$ are local martingales with quadratic variations $\langle M^{(i)} \rangle_t=\int_0^t \bigl(\mathrm{e}^{-qs}g_{i,q}'(X_{s})\bigr)^2\,\mathrm{d}s$.
By \eqref{g1g2g3} and \eqref{g1qx}, these reduce to $e^{-q t} g_{i,q}(X_{t})
=g_{i,q}(x)+ M_t^{(i)}$, and then $e^{-qt}g_{i,q}(X_t)$ are local martingales.

Let $t_n := \inf\{t\geq 0: |X_t|\geq n\}$. Since $g_{i,q}'$ is continuous on the compact set $[y,z]\cap [-n, n]$, it is bounded there: $\sup_{x\in[y, z]\cap[-n, n]} |g_{i,q}'(x)| := C_n<\infty$.
Then,
\begin{align*}
  \mathbb{E}[\langle M^{(i)} \rangle_{t_n \wedge \tau_{y,z}}]
  = \mathbb{E}\big[\int_0^{t_n\wedge \tau_{y,z}} \bigl(\mathrm{e}^{-qs}g_{i,q}'(X_{s})\bigr)^2\,\mathrm{d}s \big]
  \le C_n^2 \, \frac{1 - e^{-2q t_n}}{2q} \le \frac{C_n^2}{2q} < \infty.
\end{align*}
By It\^{o} isometry (c.f. \cite{Karatzas1991}, Proposition 3.2.10), the stochastic integrals $M_{t_n \wedge \tau_{y,z}}^{(i)}$ are square-integrable, therefore there are martingales. Equivalently, let $n\uparrow \infty$, for $x\in(y,z)$, we have
\[g_{1,q}(x)=\mathbb{E}_x [e^{-q(\tau_y\wedge \tau_z)}g_{1,q}(X_{\tau_y\wedge \tau_z})]
=g_{1,q}(y)\mathbb{E}_x [e^{-q\tau_y}; \tau_y<\tau_z]+g_{1,q}(z)\mathbb{E}_x [e^{-q\tau_z}; \tau_z<\tau_y], \]
\[g_{2,q}(x)=\mathbb{E}_x [e^{-q(\tau_y\wedge \tau_z)}g_{2,q}(X_{\tau_y\wedge \tau_z})]
=g_{2,q}(y)\mathbb{E}_x [e^{-q\tau_y}; \tau_y<\tau_z]+g_{2,q}(z)\mathbb{E}_x [e^{-q\tau_z}; \tau_z<\tau_y]. \]
The results in \eqref{e1}-\eqref{e2} are derived by solving the above system of equations.
\end{proof}

Note that $\mathbb{E}_{x}[e^{-q \tau_{y, z}}]=\mathbb{E}_{x}[e^{-q \tau_{y}}; \tau_{y}<\tau_{z}]+\mathbb{E}_{x}[e^{-q \tau_{z}}; \tau_{z}<\tau_{y}]$.
\begin{thm}\label{lap 2}
For any $q>0$, $x, y, z\in\mathbb{R}$ and $x\in(y, z)$, we have
\begin{align}\label{e3}
\mathbb{E}_{x}[e^{-q \tau_{y, z}}]
=\frac{w(x,y)-w(x,z)}{w(z,y)}.
\end{align}
In addition, $\mathbb{E}_{y}[e^{-q \tau_{y, z}}]=1$ and $\mathbb{E}_{z}[e^{-q \tau_{y, z}}]=1$.
\end{thm}

Taking the limit as $z\uparrow\infty$ or $y\downarrow-\infty$ in \eqref{e3}, we obtain the following theorem.
\begin{thm}\label{lap 3}
Given $q>0$ and $x,r\in\mathbb{R}$, we have for any $x\geq r$,
\begin{eqnarray}\label{e4}
 \mathbb{E}_{x}[e^{-q \tau_{r}}]=e^{\rho_1^+ (x-r)}\mathbf{1}_{\{r> a\}}+
 \frac{g_{1,q}(x)}{c_1(q)e^{\rho_2^-(r-a)}+\big{(}1-c_1(q)\big{)}e^{\rho_1^-(r-a)}}
 \mathbf{1}_{\{r\leq a\}},
\end{eqnarray}
and for $x<r$,
\begin{eqnarray}\label{e5}
 \mathbb{E}_{x}[e^{-q \tau_{r}}]=e^{\rho_2^- (x-r)}\mathbf{1}_{\{r\leq a\}}+
 \frac{g_{2,q}(x)}{\big{(}1-c_2(q)\big{)}e^{\rho_2^+(r-a)}+c_2(q)e^{\rho_1^+(r-a)}}
 \mathbf{1}_{\{r> a\}}.
\end{eqnarray}
\end{thm}
\begin{proof}
For $x\geq r$, letting $y=r$ and $z\uparrow\infty$ in \eqref{e3},
by L'H\^{o}pital's rule we have
\begin{align*}
\mathbb{E}_{x}[e^{-q \tau_{r}}]
&=\lim_{z\uparrow\infty}\mathbb{E}_{x}[e^{-q \tau_{r, z}}]
=\lim_{z\uparrow\infty}\frac{g_{2,q}'(z)g_{1,q}(x)-g_{1,q}'(z)g_{2,q}(x)}
{g_{2,q}'(z)g_{1,q}(r)-g_{1,q}'(z)g_{2,q}(r)}\\
&=\lim_{z\uparrow\infty}
\frac{\big((1-c_2)\rho_2^+e^{\rho_2^+ (z-a)}+c_2 \rho_1^+e^{\rho_1^+ (z-a)}\big)g_{1,q}(x)-\rho_{1}^+e^{\rho_{1}^+(z-a)} g_{2,q}(x)}
{\big((1-c_2)\rho_2^+e^{\rho_2^+ (z-a)}+c_2 \rho_1^+e^{\rho_1^+ (z-a)}\big)g_{1,q}(r)-\rho_{1}^+e^{\rho_{1}^+(z-a)} g_{2,q}(r)}.
\end{align*}
Dividing both the numerator and denominator by $e^{\rho_2^+(z-a)}$ in the above equation yields
\begin{align*}
\mathbb{E}_{x}[e^{-q \tau_{r}}]
&=\lim_{z\uparrow\infty}
\frac{\big((1-c_2)\rho_2^++c_2 \rho_1^+e^{(\rho_1^+-\rho_{2}^+) (z-a)}\big)g_{1,q}(x)-\rho_{1}^+e^{(\rho_{1}^+-\rho_{2}^+)(z-a)} g_{2,q}(x)}
{\big((1-c_2)\rho_2^++c_2 \rho_1^+e^{(\rho_1^+-\rho_{2}^+) (z-a)}\big)g_{1,q}(r)-\rho_{1}^+e^{(\rho_{1}^+-\rho_{2}^+)(z-a)} g_{2,q}(r)}\\
&=\frac{g_{1,q}(x)}{g_{1,q}(r)}.
\end{align*}
Then \eqref{e4} follows from \eqref{g1}.
Laplace transform \eqref{e5} can be obtained similarly.
\end{proof}
\begin{rem}\label{cor w-}
For any $q>0$, $y< z<a$ and $x\in(y,z)$, we have
\begin{eqnarray*}
 &&\mathbb{E}_{x}[e^{-q \tau_{z}}; \tau_{z}<\tau_{y}]=
 \frac{e^{\rho_2^-(x-y)}-e^{\rho_1^-(x-y)}}
{e^{\rho_2^-(z-y)}-e^{\rho_1^-(z-y)}},\label{ex1}\ \ \
 \mathbb{E}_{x}[e^{-q \tau_{y}}; \tau_{y}<\tau_{z}]=\frac{e^{\rho_2^-(x-z)}-e^{\rho_1^-(x-z)}}
{e^{\rho_2^-(y-z)}-e^{\rho_1^-(y-z)}}.\nonumber
\end{eqnarray*}
\end{rem}
\begin{rem}
For $\beta=0$ and $\mu_+=\mu_-=\mu$, $X$ in \eqref{JD_a} reduces to Brownian motion with drift $\mu$. In this case, for $q>0$ and $x\in(y,z)$, we have
\begin{eqnarray*}
		&&\mathbb{E}_x[e^{-q \tau_z}; \tau_z<\tau_y]
=\frac{e^{\rho_2 (x-y)}-e^{\rho_1(x-y)}}
{e^{\rho_2(z-y)}-e^{\rho_1(z-y)}},\nonumber\ \ \
\mathbb{E}_x[e^{-q \tau_y}; \tau_y<\tau_z]
=\frac{e^{\rho_2(x-z)}-e^{\rho_1(x-z)}}
{e^{\rho_2(y-z)}-e^{\rho_1(y-z)}},
\end{eqnarray*}
where $\rho_2=\rho_2^-=\rho_2^+$ and $\rho_1=\rho_1^-=\rho_1^+$.
In addition, for $q>0$ and $x, r\in \mathbb{R}$ we have
\begin{eqnarray*}
&&\mathbb{E}_x[e^{-q \tau_r}]
=e^{\rho_2(x-r)}\mathbf{1}_{\{x< r\}}+e^{\rho_1(x-r)}\mathbf{1}_{\{x\geq r\}}.
\end{eqnarray*}
These results are consistent with those presented in Part II, Section 2 of \cite{Borodin2012}.
\end{rem}

\section{Conditions for optimal dividend strategies}\label{HJB}
In this section, we first present the optimal dividend problem for refracted skew Brownian risk process, and then prove the corresponding Hamilton-Jacobi-Bellman inequalities.

A dividend strategy $\pi\equiv (D_t^{\pi})_{t\geq 0}$ is a $(\mathcal{F}_t)$-adapted process starting at $0$ with sample paths that are non-decreasing  and  right continuous with left limits, where  $D_t^{\pi}$  represents the cumulated dividends up to time $t$ under the strategy $\pi$. Define a controlled risk process $U^{\pi}\equiv(U_t^{\pi})_{t\geq 0}$ with $U_0^{\pi}=x$ by
\begin{eqnarray}\label{sde2}
	\mathrm{d} U_t^{\pi}:=\mathrm{d} B_t+\beta L^{U^{\pi}}(\mathrm{d}t, a)+(\mu_+\mathbf{1}_{\{U_t^{\pi}>a\}}
+\mu_-\mathbf{1}_{\{U_t^{\pi}< a\}})\mathrm{d} t- \mathrm{d} D_t^{\pi},
\end{eqnarray}
where $L^{U^{\pi}}(\mathrm{d}t, a)$ is defined as the local time in \eqref{loctime}. Let $T^{\pi}:=\inf\{t\geq 0: U_t^{\pi}\leq0\}$ be the ruin time.
For initial capital $x\in\mathbb{R}_+$,
the expected total amount of dividends (discounted at rate $q>0$) until ruin associated to  $\pi$ is given by
\begin{eqnarray*}\label{def V}
	V_{\pi}(x):=\mathbb{E}_x\Big[\int_0^{T^{\pi}} e^{-q t} \mathrm{d} D_t^{\pi}\Big].
\end{eqnarray*}
A strategy $\pi$ is called admissible if ruin does not occur due to dividend payments, i.e. $\Delta D_t^{\pi}=D_{t}^{\pi}-D_{t-}^{\pi}\leq U_{t-}^{\pi} \vee 0$ for any $0\leq t< T^{\pi}$, and $\Delta D_t^{\pi}=0$ for $t\geq T^{\pi}$, and $b\vee x=\max\{b,x\}$, and SDE \eqref{sde2} has a unique strong solution.
Let $\Pi$ be the set of all admissible dividend strategies. Define a value function $V_*$ by
$$V_*(x):=\sup_{\pi\in\Pi}V_{\pi}(x),\,\, x\geq 0.$$
A dividend strategy $\pi_*\in\Pi$ is optimal if $V_{\pi_*}(x)=V_*(x)$ for all $x\in \mathbb{R_+}$.

Write $\mathcal{P}:=\{p_k, k=1,\cdots,N\}$ for a fixed finite subset of $(0, \infty)\backslash \{a\}$.
Let function $f: \mathbb{R}_+\rightarrow \mathbb{R}$ be right continuous at $0$ and continuous on $(0, \infty)$. Suppose that derivatives $f'$ and $f''$ on $\mathbb{R}_+\backslash (\mathcal{P}\cup\{a\})$ are locally bounded, and both the left- and right-derivative at each $y\in \mathcal{P}\cup\{a\}$ exist.
For $y\in \mathbb{R}_+\backslash (\mathcal{P}\cup\{a\})$, define the operator $\mathcal{A}$ by
\begin{eqnarray}\label{gen}
	&&\mathcal{A} f(y):=\frac{1}{2}f''(y)+(\mu_+\mathbf{1}_{\{y>a\}}
+\mu_-\mathbf{1}_{\{y< a\}})f'(y).
\end{eqnarray}
\begin{lem}[Verification Lemma]\label{lem hjb}
\begin{enumerate}[label=(\roman*)]
\item
Suppose that $V\in C^{2}\big(\mathbb{R}_+ \backslash (\mathcal{P}\cup\{a\})\big)\cap C(\mathbb{R}_{+})$
and its first derivative has finite left- and right-limits at each $p_{k}\in\mathcal{P}$ denoted by $V'(p_{k}-)$ and $V'(p_{k}+)$, respectively.
If $V$ satisfies the following Hamilton-Jacobi-Bellman (HJB) inequalities
\begin{align}
&(\mathcal{A}-q)V(x)\leq 0 &&\mathrm{for}\ x\in \mathbb{R_+}\backslash (\mathcal{P}\cup\{a\});\label{hjb1}\\
&1-V'(x)\leq 0    &&\mathrm{for}\ x\in \mathbb{R_+}\backslash (\mathcal{P}\cup\{a\});\label{hjb2}\\
&(1+\beta) V'(a+)-(1-\beta)V'(a-)\leq 0;   \label{hjb3}\\
&V'(x+)-V'(x-)\leq 0   && \mathrm{for}\ x\in \mathcal{P}; \label{hjb4}
	\end{align}
then $V(x)\geq V_*(x)$ for all $x\in \mathbb{R_+}$.
\item
If $\hat{\pi}$ is an admissible dividend strategy with the associated expected discounted dividend function, $V_{\hat{\pi}}$, satisfying the smoothness condition and the HJB inequalities \eqref{hjb1}-\eqref{hjb4} in (i), then $V_{\hat{\pi}}(x)=V_*(x)$.
\end{enumerate}
\end{lem}
\begin{proof}
We only prove (i).
Note that $V$ can be expressed as the difference of two convex functions, c.f. Section 6 of \cite{Lejay2006}.
Then $V$ is locally Lipschitz on $\mathbb{R}_+$ and differentiable a.e., and its derivative $V'$ has locally bounded variation. Subsequently, by standard bounded variation theory, $V'$ is regulated, that is, the one-sided limits $V'(p\pm)$ exist and are finite for each $p\in \mathbb{R}_+ \cup\{a\}$ (c.f. Section 10 of \cite{Rockafellar1998}). Under our regularity assumptions $V'(p\pm)$ coincide with the one-sided derivatives of $V$ at $p$.
The second generalized derivative of $V$ is given by
\begin{eqnarray*}
\mu(\mathrm{d}y)=V{''}(y)\mathrm{d}y
+\sum_{p\in\mathcal{P}\cup\{a\}}(V'(p+)-V'(p-))\delta_p(\mathrm{d}y),
\end{eqnarray*}
where $\delta_p$ denotes a Dirac mass at $p$.
Define $T_n:=\inf\{t>0: U_t^{\pi}\notin (0, n) \}$.
Applying the It\^{o}-Tanaka-Meyer formula to
$\big(e^{-q(t\wedge T_n)}V(U_{t\wedge T_n}^{\pi})\big)_{t\geq 0}$, we have that under $\mathbb{P}_x$,
\begin{align*}
&e^{-q(t\wedge T_n)}V(U_{t\wedge T_n}^{\pi})
		=V(x)-\int_0^{t\wedge T_n}q e^{-q s} V(U_{s-}^{\pi})\mathrm{d}s+\int_0^{t\wedge T_n}e^{-q s}V'(U_{s-}^{\pi})\mathrm{d}U_{s}^{\pi}
\\ & +\frac{1}{2}\int_0^{t\wedge T_n}e^{-qs}\int_\mathbb{R_+}\mu(\mathrm{d}y)L^{U^{\pi}}(\mathrm{d}s, y)
+\sum_{0\leq s \leq t\wedge T_n}e^{-q s}\big(V(U_{s}^{\pi})-V(U_{s-}^{\pi})
-V'(U_{s-}^{\pi})\Delta U_{s}^{\pi}\big)
\\ &=V(x)-\int_0^{t\wedge T_n}q e^{-q s} V(U_{s-}^{\pi})\mathrm{d}s+\int_0^{t\wedge T_n} e^{-q s}(\mu_+\mathbf{1}_{\{U_{s-}^{\pi}>a\}}+\mu_-\mathbf{1}_{\{U_{s-}^{\pi}< a\}})V'(U_{s-}^{\pi})\mathrm{d}s
\\ &  +\int_0^{t\wedge T_n}e^{-q s}V'(U_{s-}^{\pi})\mathrm{d}B_{s}
-\int_0^{t\wedge T_n}e^{-q s}V'(U_{s-}^{\pi})\mathrm{d}D^{\pi, c}_{s}
-\sum_{0\leq s \leq t\wedge T_n}e^{-q s}V'(U_{s-}^{\pi})\Delta D_{s}^{\pi}
\\ & +\frac{\beta}{2} (V '(a+)+V'(a-))\int_0^{t\wedge T_n}e^{-q s} L^{U^{\pi}}(\mathrm{d}s, a)
+ \frac{1}{2}\int_0^{t\wedge T_n}e^{-qs}\int_\mathbb{R_+}V{''}(y)
L^{U^{\pi}}(\mathrm{d}s, y)\mathrm{d}y
\\ &
+ \frac{1}{2}\sum_{p\in \mathcal{P}\cup\{a\}}\big(V'(p+)-V'(p-)\big)
\int_0^{t\wedge T_n}e^{-q s} L^{U^{\pi}}(\mathrm{d}s, p)
\\ &
+\sum_{0\leq s \leq t\wedge T_n}e^{-q s}\big(V(U_{s}^{\pi})-V(U_{s-}^{\pi})-V'(U_{s-}^{\pi})\Delta U_{s}^{\pi}\big).
	\end{align*}
By the occupation time formula (c.f. \cite{Lejay2006}, equation (34)), for $t<T^{\pi}$, we have \[\int_\mathbb{R_+}V{''}(y)L^{U^{\pi}}(t, y)\mathrm{d}y=\int_0^t V{''}(U_{s-}^\pi)\mathrm{d}s\,\,\mathrm{and}\,\, \Delta U_{s}^{\pi}=-\Delta D_{s}^{\pi}.\]
Then, by \eqref{gen} we have
\begin{align*}
e^{-q(t\wedge T_n)}V(U_{t\wedge T_n}^{\pi})
&=V(x)+\int_0^{t\wedge T_n}e^{-q s}(\mathcal{A}-q) V(U_{s-}^{\pi})\mathrm{d}s
		+M_{t\wedge T_n}\\
&\ \ -\int_0^{t\wedge T_n}e^{-q s}V'(U_{s-}^{\pi})\mathrm{d}D^{\pi, c}_{s}
+\sum_{0\leq s \leq t\wedge T_n}e^{-q s}\big(V(U_{s}^{\pi})-V(U_{s-}^{\pi})\big)
\\&\ \
+\big(\frac{1+\beta}{2} V'(a+)-\frac{1-\beta}{2}V'(a-)\big)\int_0^{t\wedge T_n}e^{-q s} L^{U^{\pi}}(\mathrm{d}s, a)\\
&\ \ + \frac{1}{2}\sum_{p\in \mathcal{P}}\big(V'(p+)-V'(p-)\big)
\int_0^{t\wedge T_n}e^{-q s} L^{U^{\pi}}(\mathrm{d}s, p),
\end{align*}
where $(D_t^{\pi,c})_{t\geq 0}$ denotes the continuous part of the process $(D_t^{\pi})_{t\geq 0}$ and
 $M_{t}:=\int_0^t e^{-q s}V'(U_{s-}^{\pi})\mathrm{d} B_s$ for $t\geq 0$ is a local martingale with $M_0=0$.

For any $x_1, x_2\in\mathbb{R}_+$ with $x_1<x_2$, if $V\in C^1([x_1, x_2])$,
then by \eqref{hjb2} we have $V'(x)\geq 1$ for $x\in[x_1, x_2]$, and by the mean value theorem, we have
$V(x_2)-V(x_1)\geq x_2-x_1$. Similarly, if $V\in C^1([x_1, p)\cup(p, x_2])$ for $p\in\mathcal{P}\cup\{a\}$,
then
\begin{equation*}
V(x_2)-V(x_1)
=V(x_2)-V(p)+V(p)-V(x_1)\geq x_2-p+p-x_1=x_2-x_1.
\end{equation*}
Thus,
$V(U_s^{\pi})-V(U_{s-}^{\pi})
=-(V(U_{s-}^{\pi})-V(U_s^{\pi}))
\leq -(U_{s-}^{\pi}- U_s^{\pi})=\Delta U_s^{\pi}=-\Delta D_s^{\pi}.$
Combining \eqref{hjb1}, \eqref{hjb3} and \eqref{hjb4} we have
\begin{align*}
		V(x)&=e^{-q(t\wedge T_n)}V(U_{t\wedge T_n}^{\pi})-\int_0^{t\wedge T_n}e^{-q s}(\mathcal{A}-q) V(U_{s-}^{\pi})\mathrm{d}s-M_{t\wedge T_n}
\\& \ \
+\int_0^{t\wedge T_n}e^{-q s}V'(U_{s-}^{\pi})\mathrm{d}D^{\pi, c}_{s}
-\sum_{0\leq s \leq t\wedge T_n}e^{-q s}\big(V(U_{s}^{\pi})-V(U_{s-}^{\pi})\big)
\\&\ \
-\big(\frac{1+\beta}{2} V'(a+)-\frac{1-\beta}{2}V'(a-)\big)\int_0^{t\wedge T_n}e^{-q s} L^{U^{\pi}}(\mathrm{d}s, a)\\
&\ \
-\frac{1}{2}\sum_{p\in \mathcal{P}}\big(V'(p+)-V'(p-)\big)
\int_0^{t\wedge T_n}e^{-q s} L^{U^{\pi}}(\mathrm{d}s, p)\\
& \geq e^{-q(t\wedge T_n)}V(U_{t\wedge T_n}^{\pi})+\int_0^{t\wedge T_n}e^{-q s}\mathrm{d} D^{\pi, c}_{s}
+\sum_{0\leq s \leq t\wedge T_n}e^{-q s}\Delta D_s^{\pi}-M_{t\wedge T_n}\\
& = e^{-q(t\wedge T_n)}V(U_{t\wedge T_n}^{\pi})+\int_0^{t\wedge T_n}e^{-q s}\mathrm{d} D_s^{\pi}-M_{t\wedge T_n}.
	\end{align*}
Note that ${T_n}{\longrightarrow}T^\pi$ for $\mathbb{P}_x$-a.s. Taking expectation on both sides of the inequality above
and letting $t, n\uparrow\infty$, since $V\geq 0$ on $\mathbb{R}_+$, by the monotone convergence theorem
\[V(x) \geq \lim_{t,n\uparrow\infty}\mathbb{E}_x\Big[\int_0^{t\wedge T_n}e^{-q s}\mathrm{d} D_s^{\pi}\Big]=\mathbb{E}_x\Big[\int_0^{T^{\pi}}e^{-q s}\mathrm{d} D_s^{\pi}\Big]= V_{\pi}(x).\]
This completes the proof.
\end{proof}
\begin{defn}\label{defpi2}
For $\pi\in\Pi$, suppose that $V_\pi$ is piecewise $C^1$ on $\mathbb{R}_+$ with at most finitely many non-$C^1$ points. Define
$\mathcal{P}_\pi
:=\{ y \in \mathbb{R}_+ \setminus\{a\}: V_\pi \text{ is not } C^1 \text{ at } y \}$.
Then, \(\mathcal{P}_\pi\) is a finite (possibly empty) subset of $\mathbb{R}_+\setminus\{a\}$.
\end{defn}

\section{Optimal barrier strategies}\label{barrier_strategy}
In this section, we consider the barrier strategy for dividend payment.  According to such a strategy $\pi_b$ for $b\geq0$, intuitively a minimal amount of dividend is paid whenever the surplus process is going to upcross level $b$ to  keep the controlled surplus below level $b$. If the underlying surplus process is a Brownian motion with drift, the controlled process is described  by the drifted Brownian motion reflected at level $b$ and the accumulated dividends is represented by the local time for the reflected process.
For $b>a$, the refracted skew Brownian surplus process follows the dynamics of a drifted Brownian motion near $b$.
Therefore, the dividend process $ (D_t^{\pi_b})_{t\geq 0}$  is well defined and the controlled process  $U^{\pi_b}$ is the unique solution to \eqref{sde2} with $\pi$ replaced by $\pi_b$.
\subsection{Expected discounted dividend function for barrier strategies}
Denote by $\hat{\tau}_y$ the time at which process $U^{\pi_b}$ first hits the boundary $y$,
\begin{align*}
	&\hat{\tau}_{y}:=\inf\{t\geq 0, U_t^{\pi_b}\leq y\}.
\end{align*}
The time of ruin $T^{\pi_b}$ is equal to $\hat{\tau}_{0}$. We write $V_b$ for $V_{\pi_b}$ and present an expression $V_{b}$ for a general barrier dividend strategy $\pi_b$ whose proof is deferred to Appendix \ref{prof value}.
\begin{lem}\label{value function}
	For any $0\leq a_0< b$ and $b\in\mathbb{R}_+\backslash\{a\}$, we have
	\begin{align}\label{vpib1}
		&V_{b}(b,a_0):=\mathbb{E}_b\Big[\int_0^{\hat{\tau}_{a_0}} e^{-q t} \mathrm{d} D_t^{\pi_{b}}\Big]=\frac{w(b, a_0)}{w_{b}(b, a_0)},
	\end{align}
	where $w(x,y)$ is given by \eqref{wxy} and
	\begin{align}\label{wxpart}
		&w_{x}(x,y):=\partial w(x,y)/\partial x.
	\end{align}
Further, for $x\in\mathbb{R}_+$ and $b\in\mathbb{R}_+\backslash\{a\}$, we have
	\begin{eqnarray}\label{val 1}
		&V_b(x)=\left \{
		\begin{array}{ll}
			\frac{W(x)}{W'(b)} & \mathrm{for}\ 0\leq x\leq b, \\
x-b+\frac{W(b)}{W'(b)} & \mathrm{for}\ x>b,
		\end{array} \right.
	\end{eqnarray}
where
\begin{eqnarray}\label{sca 1}
	&&W(x):=w(x,0)=g_{2,q}(x)g_{1,q}(0)-g_{1,q}(x)g_{2,q}(0).
	\end{eqnarray}
\end{lem}

Notice that, for $0\leq x< a$ and $\beta=0$, by \eqref{sca 1} and \eqref{1c1q1} we have
\begin{align*}
&W(x)=\big(1-c_1(q)\big)e^{-(\rho_1^-+\rho_2^-)a}
(e^{\rho_2^-x}-e^{\rho_1^-x})=\frac{(\rho_2^- -\rho_1^+)e^{2\mu_-a}}{\rho_2^--\rho_1^-} (e^{\rho_2^- x}-e^{\rho_1^- x}),
\end{align*}
which is proportional to the scale function $\frac{2}{\rho_2^--\rho_1^-}(e^{\rho_2^- x}-e^{\rho_1^- x})$ of the classical Brownian motion with a drift $\mu_-$.
In addition, $V_{a-}$ and $V_{a+}$ denote the limits as $b$ approaches $a$ from the left and right, respectively, i.e. $V_{a-}=\lim_{b\rightarrow a-} V_b$ and $V_{a+}=\lim_{b\rightarrow a+} V_b$.

\begin{rem}
Because the	Skew Brownian motion has a singular drift at level $a$, its behavior at $a$ is very different from the standard diffusion. For example, its transition density function is not continuous at $a$; see \cite{Ahmadi2024}.
Due to the skewness, the scale function $W$ is continuous, but not differentiable at $a$, and as a result, $V_{a-}\neq V_{a+}$ in general.
\end{rem}

\begin{rem}\label{brownone}
For $\beta=0$ and $\mu_+=\mu_-=\mu$, $X$ in \eqref{JD_a} reduces to Brownian motion with drift $\mu$. In this case, the function \eqref{val 1} simplifies to
\begin{eqnarray*}
		&V_b(x)=\left \{
		\begin{array}{ll}
			\frac{e^{\rho_2 x}-e^{\rho_1 x}}{\rho_2 e^{\rho_2 b}- \rho_1 e^{\rho_1 b}} & \mathrm{for}\ 0\leq x\leq b, \\
x-b+\frac{e^{\rho_2 b}-e^{\rho_1 b}}{\rho_2 e^{\rho_2 b}- \rho_1 e^{\rho_1 b}} & \mathrm{for}\ x>b,
		\end{array} \right.
	\end{eqnarray*}
where $\rho_2=\rho_2^-=\rho_2^+$ and $\rho_1=\rho_1^-=\rho_1^+$, consistent with (2.11) of \cite{Gerber2004}.
\end{rem}
\subsection{Optimal barrier strategies}
By \eqref{val 1},
to identify the optimal barrier strategy $\pi_*$ that maximizes $V_b$ for any given $x\geq 0$, one needs to discuss the convexity and extreme behaviour of the function $W'$.
We next present expressions of $W'(x)$ for $x\geq 0$. Its proof is deferred to Appendix \ref{prop ww11}.
\begin{prop}\label{prop ww}
For any $0\leq x< a$, we have
\begin{align}\label{wa-}
&W'(x)=\big(1-c_1(q)\big)e^{-(\rho_1^-+\rho_2^-)a}
(\rho_2^-e^{\rho_2^-x}-\rho_1^-e^{\rho_1^-x}).
\end{align}
For any $x>a>0$, we have
\begin{align}\label{wa221}
W'(x)&=\rho_2^+\Big(c_1(q)e^{-\rho_2^-a}+\big(1-c_1(q)\big)e^{-\rho_1^-a}\Big)
\big(1-c_2(q)\big)e^{\rho_2^+(x-a)}\\
&\ \  -\rho_1^+\Big(\big(1-c_1(q)c_2(q)\big)e^{-\rho_2^-a}-c_2(q)\big(1-c_1(q)\big)e^{-\rho_1^-a}\Big)
e^{\rho_1^+(x-a)}.\nonumber
\end{align}
In particular,
\begin{align}\label{wa-a+}
(1-\beta)W'(a-)=(1+\beta) W'(a+).
\end{align}
\end{prop}
Since $\lim_{b\uparrow\infty}W'(b)=\infty$, we have $\lim_{b\uparrow\infty}V_b(x)=0$, and then, $V_b(x)$ attains its maximum for a finite value of $b\geq 0$.
To determine the extreme behaviour of $W'$ over intervals $(0,a)$ and $(a,\infty)$, we consider the possible solutions of equation $W''(x)=0$ in the two intervals. For $0\leq x< a$,
\begin{eqnarray}\label{b11}
&&W''(x)=\big(1-c_1(q)\big)e^{-(\rho_1^-+\rho_2^-)a}\big({\rho_2^-}^2e^{\rho_2^-x}
-{\rho_1^-}^2e^{\rho_1^-x}\big),
\end{eqnarray}
and for $x>a>0$,
\begin{align}\label{b22}
W''(x)&={\rho_2^+}^2\big(1-c_2(q)\big)\Big(c_1(q)e^{-\rho_2^-a}+\big(1-c_1(q)\big)e^{-\rho_1^-a}\Big)
e^{\rho_2^+(x-a)}\\
&\ \  -{\rho_1^+}^2 \Big(\big(1-c_1(q)c_2(q)\big)e^{-\rho_2^-a}-c_2(q)\big(1-c_1(q)\big)e^{-\rho_1^-a}\Big)
e^{\rho_1^+(x-a)}.\nonumber
\end{align}
\begin{prop}\label{propb-b+K}
Equation $W''(x)=0, x\in\mathbb{R}$, has a unique solution $b_-$ for $W''$ given by \eqref{b11}, and a unique solution $b_+$ if $K(\beta)>0$ for $W''$ given by \eqref{b22}, where
\begin{eqnarray}
&& b_-:=\frac{2}{\rho_2^--\rho_1^-}\ln\frac{-\rho_1^-}{\rho_2^-}, \label{b-}\\
&& b_+:=a+\frac{1}{\rho_2^+-\rho_1^+}\ln K(\beta), \label{b+} \\
&& K(\beta):=\frac{{\rho_1^+}^2\big((1-c_1(q)c_2(q))e^{-\rho_2^-a}-c_2(q)(1-c_1(q))
e^{-\rho_1^-a}\big)}{{\rho_2^+}^2(1-c_2(q))(c_1(q)e^{-\rho_2^-a}+(1-c_1(q))e^{-\rho_1^-a})} \label{K}.
\end{eqnarray}
\end{prop}

Its proof is deferred to Appendix \ref{propb-b+K11}. Notice that $b_-$ and $b_+$ do not depend on the initial surplus $x\geq0$. In addition, we have the following proposition. Its proof is deferred to Appendix \ref{b11min1}.
\begin{prop}\label{b1min}
\begin{enumerate}[label=(\roman*)]
\item $b_-\leq0$ if and only if $\mu_-\leq 0$.
\item $b_+>a$ if and only if $K(\beta)>1$.
\end{enumerate}
\end{prop}

We are interested in those solutions such that $b_-\in (0, a)$ and $b_+\in (a, \infty)$. The following Lemma summarizes the monotone and convex behaviours of $W'$.
Its proof is deferred to Appendix \ref{a.2}.
\begin{lem}[Monotonicity and convexity for $W$ and $W'$]\label{lem W}
For any $\beta\in(-1,1)$, the function $W$ is a non-negative continuous increasing function  on $\mathbb{R}_+$ that is further twice continuously differentiable on $\mathbb{R}_+\backslash\{a\}$. Its derivative $W'$ satisfies $W'(x)>0$ for $x\in \mathbb{R}_+\backslash\{a\}$, and its convexity and monotonicity is summarized below.
\begin{enumerate}[label=(\Roman*)]
\item For $0\leq x< a$, we have
\begin{enumerate}[label=(\roman*)]
\item $W'(x)$ is strictly increasing if and only if $b_-\leq0$,
\item $W'(x)$ is non-monotone convex if and only if $0<b_-<a$,
\item $W'(x)$ is strictly decreasing if and only if $a\leq b_-$.
\end{enumerate}
\item For $x>a>0$, we have
\begin{enumerate}[label=(\roman*)]
\item $W'(x)$ is strictly increasing if  $K(\beta)\leq1 $,
\item $W'(x)$ is non-monotone convex if $K(\beta)>1$.
\end{enumerate}
\end{enumerate}
In conclusion, the restriction of $W'$ to $(0, a)$ has a unique local minimum at $b_-$ if and only if $\mu_-> 0$ and $b_-<a$, and the restriction of $W'$ to $(a,\infty)$ has a unique local minimum at $b_+$ if and only if $K(\beta)>1$.
\end{lem}

Combining Lemmas \ref{value function} and \ref{lem W}, we obtain the following results on continuity and differentiability of function $V_b$.
\begin{rem}
For $b<a$ we have $V_b\in C(\mathbb{R_+})\cap C^2(\mathbb{R_+})$. For $b>a$ we have $V_b\in C(\mathbb{R_+})\cap C^2(\mathbb{R_+}\backslash \{a\})$. Further, $\mathcal{P}_{\pi_b}=\emptyset$ for both $b<a$ and $b>a$.
\end{rem}

Lemma \ref{lem W} suggests that $W'(x)$ may have its minimum at
\begin{align}\label{mathcalmin}
&\mathcal M:=\operatorname*{arg\,min}_{x\in\mathbb R_+\setminus\{a\}} W'(x)\in\{0+, b_-, a-, a+, b_+\}.
\end{align}
We thus propose five corresponding barrier strategies, and apply Lemma \ref{lem hjb} to identify conditions for each of the above barrier strategies to be optimal. We next introduce a proposition concerning $W'$ as function of $\beta$, and  write $W_{\beta}'$ for $W'$ to stress its dependence on $\beta$. Its proof can be found in Appendix \ref{a.4}.
\begin{prop}\label{prop W1}
	$W_{0}'(a-)=W_{0}'(a+)$ and for $\beta\neq 0$,
$W_{\beta}'(a-)<W_{\beta}'(a+)$ if and only if $\beta\in(-1, 0)$.
\end{prop}

We now present the main results of this section.
\begin{thm}\label{thm v0+}
(Optimality of $0$-barrier)
If all of the following three conditions hold
\begin{enumerate}[label=(\roman*)]
\centering
\item $\mu_-\leq0$,\ \ (\romannumeral2) $\mu_+\leq q a$, \ \ (\romannumeral3) $\beta\in(-1, 0]$,
\end{enumerate}
then the function $V_{0}$ satisfies the HJB inequalities \eqref{hjb1}-\eqref{hjb4}, and $V_*=V_{0}$ and $\pi_*=\pi_{0}$.
\end{thm}
\begin{proof}
By \eqref{val 1} we have $V_{0}(x)=x$, and then
$V_{0}'(x)=1$ and $V_{0}''(x)=0$ for $x\in\mathbb{R}_+$. We next verify the HJB inequalities \eqref{hjb1}-\eqref{hjb4}. For $0\leq x<a$, by condition (\romannumeral1) we  obtain that
$$\frac{1}{2}V_{0}''(x)+\mu_-V_{0}'(x)-q V_{0}(x)
=\mu_--q x\leq \mu_-\leq 0.$$
For $x>a$, by condition (\romannumeral2) we get
$$\frac{1}{2}V_{0}''(x)+\mu_+V_{0}'(x)-q V_{0}(x)
=\mu_+-q x< \mu_+-q a\leq0.$$
Combining the above we have \eqref{hjb1} holds.
Since $V_{0}'(x)=1$, we have $1-V_{0}'(x)=0$ for $x\geq 0$, and then \eqref{hjb2} holds. In addition, by  condition (\romannumeral3) we have
$$(1+\beta)V_{0}'(a+)-(1-\beta)V_{0}'(a-)=2\beta\leq0.$$
Thus, \eqref{hjb3} holds. Since $\mathcal{P}_{\pi_0}=\emptyset$, we have \eqref{hjb4} holds. Therefore, $V_0$ satisfies the HJB inequalities \eqref{hjb1}-\eqref{hjb4} under conditions (\romannumeral1)-(\romannumeral3).
\end{proof}
\begin{rem}[Necessary condition for optimal $0$-barrier]\label{rem v02}
The optimality of $V_{0}$ implies that $W'$ attains its minimum at $0$. By Lemma \ref{lem W} (I) we have, the necessary condition for this is $b_-\leq0$ under which  $W'(x)$ is strictly increasing for $x\in[0,a)$, and then by Proposition \ref{b1min} (i) we get $\mu_-\leq0$.
\end{rem}
\begin{thm}\label{thm vb-}
(Optimality of $b_-$-barrier)
If all of the following three conditions hold,
\begin{enumerate}[label=(\roman*)]
\item $b_-\in(0, a)$,
\item $\big(\mu_+-q(a-b_-)\big)W'(b_-)\leq qW(b_-)$,
\item  $\beta\in(-1, 0]$,
\end{enumerate}
then the function $V_{b_-}$ satisfies the HJB inequalities \eqref{hjb1}-\eqref{hjb4}, and $V_*=V_{b_-}$ and $\pi_*=\pi_{b_-}$.
\end{thm}
\begin{proof}
We now show that $V_{b_-}$ satisfies the HJB inequality \eqref{hjb1}. For $0\leq x\leq b_-$,
\begin{eqnarray*}
&&W(x)=\big(1-c_1(q)\big)e^{-(\rho_1^-+\rho_2^-)a}(e^{\rho_2^-x}-e^{\rho_1^-x}),
\end{eqnarray*}
and then, by \eqref{dfgh}, \eqref{wa-} and \eqref{b11} we have
\begin{eqnarray*}
&&\frac{1}{2}V_{b_-}''(x)+\mu_-V_{b_-}'(x)-qV_{b_-}(x)\\
&&=\frac{\big(1-c_1(q)\big)e^{-(\rho_1^-+\rho_2^-)a}}{W'(b_-)}
\big((\frac{1}{2}{\rho_2^-}^2+\mu_-\rho_2^--q)e^{\rho_2^-x}
-(\frac{1}{2}{\rho_1^-}^2+\mu_-{\rho_1^-}-q)e^{\rho_1^-x}\big)
=0.\nonumber
\end{eqnarray*}
In particular, $V_{b_-}'(b_-)=1$. By condition (\romannumeral1) and the definition of $b_-$, we have $W''(b_-)=0$ and then $V_{b_-}''(b_-)=0$. Thus, we have
\begin{align}\label{mulg}
&V_{b_-}(b_-)=\frac{\mu_-}{q}.
\end{align}
For $b_-<x<a$, combining the fact that $V_{b_-}'(x)=1$, $V_{b_-}''(x)=0$ and $V_{b_-}(x)>V_{b_-}(b_-)$, we get
$$\frac{1}{2}V_{b_-}''(x)+\mu_-V_{b_-}'(x)-q V_{b_-}(x)
=\mu_--q V_{b_-}(x)\leq \mu_--q V_{b_-}(b_-)=0.$$
For $x>a$, by \eqref{val 1} we get $V_{b_-}'(x)=1$ and $V_{b_-}''(x)=0$, and since $V_{b_-}(x)>V_{b_-}(a)$, by condition (\romannumeral2) we have
\begin{align*}
\frac{1}{2} V_{b_-}''(x)+\mu_+V_{b_-}'(x)-q V_{b_-}(x)
&=\mu_+-q V_{b_-}(x)<\mu_+-q V_{b_-}(a)
\leq0.
\end{align*}
So, \eqref{hjb1} holds.
To prove inequality \eqref{hjb2}, for $0\leq x\leq b_-$, by condition (\romannumeral1), from Lemma \ref{lem W} (I) it follows that $W'(x)$ is non-monotone convex for $x\in[0, a)$ and $\inf_{x\in[0,a)} W'(x)=W'(b_-)$, then,
$$1-V_{b_-}'(x)=1-\frac{W'(x)}{W'(b_-)}\leq 1-\frac{W'(b_-)}{W'(b_-)}=0.$$
For $x\in(b_-, a)\cup(a, \infty)$, from $V_{b_-}'(x)=1$ we can directly obtain $1-V_{b_-}'(x)=0$. Thus, \eqref{hjb2} holds. In addition, since $V_{b_-}'(a+)=V_{b_-}'(a-)=1$, by condition (\romannumeral3) we have
$$(1+\beta)V_{b_-}'(a+)-(1-\beta)V_{b_-}'(a-)=2\beta\leq0.$$
Thus, \eqref{hjb3} holds. Since $\mathcal{P}_{\pi_{b_-}}=\emptyset$, we have \eqref{hjb4} holds. Therefore, $V_{b_-}$ satisfies the HJB inequalities \eqref{hjb1}-\eqref{hjb4} under conditions (\romannumeral1)-(\romannumeral3).
\end{proof}
\begin{rem}[Sufficient condition for optimal $b_-$-barrier]\label{rem vb-1}
If $0< b_-<a$, $\mu_+\leq0$ and $\beta\in(-1, 0]$, then conditions (\romannumeral1)-(\romannumeral3) in Theorem \ref{thm vb-} are satisfied, and $V_*=V_{b_-}$.
\end{rem}
\begin{proof}
Conditions (i) and (iii) in Theorem \ref{thm vb-} are satisfied. By Lemma \ref{lem W} we have $W(b_-), W'(b_-)>0$. Since $a-b_->0$, if $\mu_+\leq 0$, then
$\big(\mu_+-q(a-b_-)\big)W'(b_-)<0< qW(b_-)$, and condition (ii) in Theorem \ref{thm vb-} holds. Therefore, $V_*=V_{b_-}$.
\end{proof}
\begin{rem}[Necessary condition for optimal $b_-$-barrier]\label{rem vb-2}
The optimality of $V_{b_-}$ requires that $W'$ attains its minimum at $b_-$. By Lemma \ref{lem W} (I) we have, the necessary condition for this is $0< b_-<a$ under which $W'(x)$ is non-monotone convex for $x\in[0, a)$.
\end{rem}
\begin{thm}\label{thm va-}
(Optimality of $V_{a-}$) If all of the following three conditions hold,
\begin{enumerate}[label=(\roman*)]
\centering
\item $b_-\geq a$,\ \
(\romannumeral2) $\mu_+W'(a-)\leq qW(a)$,\ \
(\romannumeral3) $\beta\in(-1, 0]$,
\end{enumerate}
then the function $V_{a-}$ satisfies the HJB inequalities \eqref{hjb1}-\eqref{hjb4}, and $V_*=V_{a-}$.
\end{thm}
\begin{proof}
We now verify whether $V_{a-}$ satisfies the HJB inequality \eqref{hjb1}.
For $0\leq x<a$, using a similar method to the proof in Theorem \ref{thm vb-} for $0\leq x \leq b_-$, we can obtain
\begin{eqnarray*}
&&\frac{1}{2}V_{a-}''(x)+\mu_-V_{a-}'(x)-qV_{a-}(x)=0.
\end{eqnarray*}
Clearly $V_{a-}'(a-)=1$. For $x>a$, from \eqref{val 1} it follows that $V_{a-}'(x)=1$, $V_{a-}''(x)=0$ and $V_{a-}(x)>V_{a-}(a)$, and then by condition (\romannumeral2) we have
$$\frac{1}{2}V_{a-}''(x)+\mu_+V_{a-}'(x)-q V_{a-}(x)
=\mu_+-q V_{a-}(x)\leq \mu_+-q V_{a-}(a)
=\mu_+-q \frac{W(a)}{W'(a-)}
\leq 0.$$
So \eqref{hjb1} holds. Next, we proceed to prove inequality \eqref{hjb2}. For $0\leq x< a$, by condition (\romannumeral1), from Lemma \ref{lem W} (I) it follows that $W'(x)$ is strictly decreasing for $x\in[0, a)$ and $\inf_{x\in[0,a)} W'(x)=W'(a-)$, and then
 $$1-V_{a-}'(x)=1-\frac{W'(x)}{W'(a-)}\leq 1-\frac{W'(a-)}{W'(a-)}=0.$$
For $x>a$, from $V_{a-}'(x)=1$, we can directly obtain $1-V_{a-}'(x)=0$. Thus, \eqref{hjb2} holds.
Moreover, since $V_{a-}'(a-)=V_{a-}'(a+)=1$, by condition (\romannumeral3) we get
$$(1+\beta)V_{a-}'(a+)-(1-\beta)V_{a-}'(a-)
=(1+\beta)-(1-\beta)=2\beta\leq0.$$
Then, \eqref{hjb3} holds. Since $\mathcal{P}_{\pi_{a-}}=\emptyset$, we have \eqref{hjb4} holds. Therefore, the HJB inequalities \eqref{hjb1}-\eqref{hjb4} hold for $V_{a-}$ under conditions (\romannumeral1)-(\romannumeral3).
\end{proof}
\begin{rem}
Although $V_*=V_{a-}$ in Theorem \ref{thm va-}, we can not find the corresponding optimal strategy $\pi_*$. Instead,  the collection of barrier strategies $(\pi_{a-1/n})_n$ is ``asymptotically optimal'' in the sense that $V_{a-}=\lim_{n\to\infty}V(a-1/n)$.
\end{rem}
\begin{rem}[Sufficient condition for optimality of $V_{a-}$]\label{rem va-1}
If $b_-\geq a$, $\mu_+\leq0$ and $\beta\in(-1, 0]$, then conditions (\romannumeral1)-(\romannumeral3) in Theorem \ref{thm va-} are satisfied, and $V_*=V_{a-}$.
\end{rem}
\begin{rem}[Necessary condition for optimality of  $V_{a-}$]\label{rem va-2}
The optimality of $V_{a-}$ indicate that $W'(a-)=\inf_{x\in[0, a)}W'(x)$. By Lemma \ref{lem W} (I) we have, the necessary condition for this is $a\leq b_-$ under which $W'(x)$ is strictly decreasing for $x\in[0,a)$.
\end{rem}
\begin{thm}\label{thm va+}
(Optimality of $V_{a+}$)
If both of the following two conditions hold,
\begin{enumerate}[label=(\roman*)]
\centering
\item $\inf_{x\in[0,a)} W'(x)\geq W'(a+)$,\ \ (\romannumeral2) $\mu_+W'(a+)\leq qW(a)$,
\end{enumerate}
then the function $V_{a+}(x)$ satisfies the HJB inequalities \eqref{hjb1}-\eqref{hjb4}, and $V_*=V_{a+}$.
\end{thm}
\begin{proof}
We now consider whether $V_{a+}$ satisfies the HJB inequality \eqref{hjb1}. For $0\leq x<a$, following the proof used in Theorem \ref{thm vb-} for $0\leq x \leq b_-$, we have
$$\frac{1}{2}V_{a+}''(x)+\mu_-V_{a+}'(x)-qV_{a+}(x)=0.$$
For $x>a$, from \eqref{val 1} it follows that $V_{a+}'(x)=1$, $V_{a+}''(x)=0$ and $V_{a+}(x)>V_{a+}(a)$. Particularly $V_{a+}'(a+)=1$ and $V_{a+}''(a+)=0$. By condition (\romannumeral2) we have
$$\frac{1}{2}V_{a+}''(x)+\mu_+V_{a+}'(x)-q V_{a+}(x)
=\mu_+-q V_{a+}(x)
\leq \mu_+-q V_{a+}(a+)=\mu_+-q \frac{W(a)}{W'(a+)}
\leq0.$$
So \eqref{hjb1} holds. We next show that inequality \eqref{hjb2} holds. For $0\leq x< a$, the condition (\romannumeral1) implies that $W'(x)\geq W'(a+)$, and then,
 $$1-V_{a+}'(x)=1-\frac{W'(x)}{W'(a+)}\leq 1-\frac{W'(a+)}{W'(a+)}=0.$$
For $x>a$, since $V_{a+}'(x)=1$, we have $1-V_{a+}'(x)=0$. Thus, \eqref{hjb2} holds.
In addition, by \eqref{wa-a+} we have
$(1+\beta)V_{a+}'(a+)-(1-\beta)V_{a+}'(a-)
=(1+\beta)-(1+\beta)=0.$
Thus, \eqref{hjb3} holds. Since $\mathcal{P}_{\pi_{a+}}=\emptyset$, we have \eqref{hjb4} holds. Therefore, the HJB inequalities \eqref{hjb1}-\eqref{hjb4} hold for $V_{a+}$ under conditions (\romannumeral1)-(\romannumeral2).
\end{proof}
\begin{rem}[Sufficient condition for optimality of $V_{a+}$]\label{rem va+1}
If $a\leq b_-$, $\mu_+\leq0$ and $\beta\in[0, 1)$, then conditions (\romannumeral1)-(\romannumeral2) in Theorem \ref{thm va+} are satisfied, and $V_*=V_{a+}$.
\end{rem}
\begin{proof}
When $a\leq b_-$, by Lemma \ref{lem W} (I) we have, $W'(x)$ is strictly decreasing for $x\in [0, a)$, and then $W'(a-)=\inf_{x\in[0,a)}W'(x)$. By Proposition \ref{prop W1} we get $W'(a+)\leq W'(a-)$ for $\beta\in[0, 1)$. Thus $W'(a+)\leq\inf_{x\in[0,a)}W'(x)$, i.e. condition (\romannumeral1) in Theorem \ref{thm va+} holds. By Lemma \ref{lem W} we have $W(a), W'(a+)>0$. If $\mu_+\leq 0$, then
$\mu_+W'(a+)\leq0< qW(a)$,
and condition (\romannumeral2) in Theorem \ref{thm va+} holds. Therefore, $V_*=V_{a+}$.
\end{proof}
\begin{rem}[Necessary condition for optimality of $V_{a+}$]\label{rem va+2}
The optimality of $V_{a+}$ implies $W'(a+)\leq \inf_{x\in[0, a)}W'(x)$, and then $W'(a+)\leq W'(a-)$. By Proposition \ref{prop W1} we have, the necessary condition for this is $\beta\in[0, 1)$.
\end{rem}
\begin{thm}\label{thm vb2}
(Optimality of $b_+$-barrier)
If both of the following two conditions hold,
\begin{enumerate}[label=(\roman*)]
\centering
\item $b_+>a$,\ \ (\romannumeral2) $\inf_{x\in\mathbb{R_+}\backslash\{a\}} W'(x)=W'(b_+)$,
\end{enumerate}
then the function $V_{b_+}$ satisfies the HJB inequalities \eqref{hjb1}-\eqref{hjb4} with $V_*=V_{b_+}$ and $\pi_*=\pi_{b_+}$.
\end{thm}
\begin{proof}
We now show that $V_{b_+}$ satisfies the HJB inequality \eqref{hjb1}. For $0\leq x< a$, similar to the proof in Theorem \ref{thm vb-} for $0\leq x \leq b_-$, we can obtain
\begin{eqnarray*}
&&\frac{1}{2}V_{b_+}''(x)+\mu_-V_{b_+}'(x)-qV_{b_+}(x)
=0.
\end{eqnarray*}
For $a<x\leq b_+$,
\begin{align*}
W(x)&=\big(c_1(q) e^{-\rho_2^-a}+(1-c_1(q)) e^{-\rho_1^-a}\big)(1-c_2(q))e^{\rho_2^+(x-a)}\\
&\ \ \ -\big((1-c_1(q)c_2(q))e^{-\rho_2^-a}-c_2(q) (1-c_1(q)) e^{-\rho_1^-a}\big)e^{\rho_1^+(x-a)}.
\end{align*}
By \eqref{dfgh212}, \eqref{wa221} and \eqref{b22} we have
\begin{eqnarray*}
&&\frac{1}{2}V_{b_+}''(x)+\mu_+V_{b_+}'(x)-qV_{b_+}(x)=0.
\end{eqnarray*}
By condition (i) and the definition of $b_+$, we get $W''(b_+)=0$, and then $V_{b_+}''(b_+)=0$. From $V_{b_+}'(b_+)=1$ it follows that $V_{b_+}(b_+)=\frac{\mu_+}{q}$. For $x>b_+$, based on $V_{b_+}'(x)=1$ and $V_{b_+}''(x)=0$, using the fact that $V_{b_+}(x)>V_{b_+}(b_+)$, we get
\begin{eqnarray*}
\frac{1}{2}V_{b_+}''(x)+\mu_+V_{b_+}'(x)-q V_{b_+}(x)
=\mu_+-q V_{b_+}(x)< \mu_+-q V_{b_+}(b_+)=0.
\end{eqnarray*}
So, \eqref{hjb1} holds. Besides that, by condition (ii),
for $x\in[0,a)\cup(a,b_+]$, we have
 $$1-V_{b_+}'(x)=1-\frac{W'(x)}{W'(b_+)}\leq 1-\frac{W'(b_+)}{W'(b_+)}=0.$$
For $x>b_+$, since $V_{b_+}'(x)=1$, we have $1-V_{b_+}'(x)=0$. Thus, \eqref{hjb2} holds.
By \eqref{wa-a+} we get
$$(1+\beta) V_{b_+}'(a+)-(1-\beta) V_{b_+}'(a-)
={(W'(b_+))}^{-1}\big((1+\beta) W'(a+)-(1-\beta)W'(a-)\big)=0.$$
Thus, \eqref{hjb3} holds. Since $\mathcal{P}_{\pi_{b_+}}=\emptyset$, we have \eqref{hjb4} holds. Therefore, the HJB inequalities \eqref{hjb1}-\eqref{hjb4} hold for $V_{b_+}$ under conditions (\romannumeral1)-(\romannumeral2).
\end{proof}

\section{Optimal band strategies}\label{band_strategy}
In this section we consider a class of band strategies, denoted by $\pi_{b_1, a_1, b_2}$ for $0\leq b_1\leq a_1\leq b_2$, that involve two dividend barriers at levels $b_1$ and $b_2$, respectively. Such a dividend strategy can be described as follows. If the surplus level is above $b_2$, a lump-sum payment is made to bring the surplus level to $b_2$. If the surplus takes values in $(a_1, b_2]$, a dividend barrier at level $b_2$ is imposed until the surplus first reaches level $a_1$. If the surplus takes values in $(b_1, a_1]$, a lump-sum payment is made to bring the surplus  to $b_1$. If the surplus takes values in $(0, b_1]$, a dividend barrier at level $b_1$ is imposed until ruin occurs.                                                                                                                                                                                                                                                                                                                                                                                                                                                                                                                                                                                                                                                                                                                                                                                                                                                                                                                                                                                             We refer to \cite{Azcue2005} for introductions on band strategies.
Write $V_{b_1, a_1, b_2}(x)$ for the expected total amount of discounted dividends before ruin with band strategy $\pi_{b_1, a_1, b_2}$ and $X_0=x$, and replace $\mathcal{P}_{\pi}$ as defined in Definition \ref{defpi2}, with $\mathcal{P}_{\pi_{b_1, a_1, b_2}}$.
\subsection{Expected discounted dividend function for band strategies}
\begin{lem}\label{thm vbc2}
	For $x\in\mathbb{R}_+$, $0\leq b_1< a< b_2$ and $0\leq b_1\leq a_1\leq b_2$, we have
\begin{eqnarray*}
		&V_{b_1,a_1, b_2}(x)=\left \{
		\begin{array}{ll}
			\frac{W(x)}{W'(b_1)} & \mathrm{for}\ x\in[0, b_1), \\
x-b_1+\frac{W(b_1)}{W'(b_1)} & \mathrm{for}\ x\in[b_1, a_1), \\
 \frac{w(x, a_1)}{w_{b_2}(b_2, a_1)}+\big(a_1-b_1+\frac{W(b_1)}{W'(b_1)}\big)
 \frac{w_{b_2}(b_2, x)}{w_{b_2}(b_2, a_1)} & \mathrm{for}\ x\in[a_1, b_2), \\
x-b_2+\frac{w(b_2, a_1)}{w_{b_2}(b_2, a_1)}+\big(a_1-b_1+\frac{W(b_1)}{W'(b_1)}\big) \frac{w_{b_2}(b_2, b_2)}{w_{b_2}(b_2, a_1)} & \mathrm{for}\ x\in[b_2, \infty),
		\end{array} \right.
	\end{eqnarray*}
where $w(x,y)$, $w_x(x,y)$ and $W(x)$ are defined by \eqref{wxy}, \eqref{wxpart} and \eqref{sca 1}, respectively.
\end{lem}

Its proof is deferred to Appendix \ref{thm vbc21}.
For $x\in[0,a_1 \wedge a)$, the monotonicity of $W'(x)$ has been described in Lemma \ref{lem W} (I), and then,
the monotonicity of $V_{b_1,a_1, b_2}'(x)$ for $x\in[0,a_1)$ is determined. We first present the following proposition, which will be used in the subsequent analysis of the monotonicity of $V_{b_1,a_1, b_2}'(x)$ for $x\in(a_1, \infty)$. Its proof is deferred to Appendix \ref{wb2b2a1u1}.
\begin{prop}\label{wb2b2a1u2}
For any $0< a< a_1\leq b_2$, we have
\begin{align}\label{wb2b1fu3}
w_{b_2}(b_2, a_1)
&=\big(1-c_{2}(q)\big)e^{(\rho_2^++\rho_1^+)(a_1-a)}
\big(\rho_2^+ e^{\rho_2^+(b_2-a_1)}-\rho_1^+ e^{\rho_1^+(b_2-a_1)}\big).
\end{align}
For any $0\leq a_1\leq a< b_2$, we have
\begin{align}\label{wb2b1fu1}
w_{b_2}(b_2, a_1)
&=\rho_2^+ \big(1-c_{2}(q)\big)\Big(c_{1}(q)e^{\rho_2^-(a_1-a)}+\big(1-c_{1}(q)\big)e^{\rho_1^-(a_1-a)}\Big)
e^{\rho_2^+(b_2-a)}\\
&\ \ -\rho_1^+\Big(\big(1-c_{1}(q)c_{2}(q)\big)e^{\rho_2^-(a_1-a)}
-c_{2}(q)\big(1-c_{1}(q)\big)e^{\rho_1^-(a_1-a)}\Big) e^{\rho_1^+(b_2-a)}.\nonumber
\end{align}
Further, $w_{b_2}(b_2, a_1)>0$ for any $0< a<b_2$ and $0\leq a_1\leq b_2$.
\end{prop}
\subsection{Optimal band strategies}
We now seek to identify the optimal band strategy. For $x\in[0, b_1)$, by Lemma \ref{thm vbc2} we have
\begin{align*}
 V_{b_1,a_1, b_2}''(x)&=\frac{W''(x)}{W'(b_1)}.
\end{align*}
For $x\in[a_1\wedge a, a)$, by Lemma \ref{thm vbc2} we have
\begin{align}
V_{b_1,a_1, b_2}'(x)
 &=\frac{\tilde{K}_{2}(b_1, a_1, b_2){\rho_2^-}e^{\rho_2^-(x-a)}-\tilde{K}_{1}(b_1, a_1, b_2){\rho_1^-}e^{\rho_1^-(x-a)}}{w_{b_2}(b_2, a_1)},\label{vb1a1b266}\\
V_{b_1, a_1, b_2}''(x)&=
\frac{\tilde{K}_{2}(b_1, a_1, b_2){\rho_2^-}^2e^{\rho_2^-(x-a)}-\tilde{K}_{1}(b_1, a_1, b_2){\rho_1^-}^2e^{\rho_1^-(x-a)}}{w_{b_2}(b_2, a_1)},\label{vsecde}
\end{align}
where $g_{1,q}(x)$ and $g_{2,q}(x)$ are given by \eqref{g1} and \eqref{g2}, respectively, and
\begin{align}
&\tilde{K}_{1}(b_1, a_1, b_2):=(1-c_{1}(q))\big(g_{2,q}(a_1)-g_{2,q}'(b_2)V_{b_1,a_1, b_2}(a_1)\big),\label{k1b1a1}\\
&\tilde{K}_{2}(b_1, a_1, b_2)
:=g_{1,q}(a_1)-g_{1,q}'(b_2)V_{b_1,a_1, b_2}(a_1)
-c_{1}(q)\big(g_{2,q}(a_1)-g_{2,q}'(b_2)V_{b_1,a_1, b_2}(a_1)\big).\label{k1b1a2}
\end{align}
For $x\in(a_1 \vee a, b_2)$, by Lemma \ref{thm vbc2} we have
\begin{align}\label{vb1a1b222}
V_{b_1,a_1, b_2}''(x)
 &=
\frac{\hat{K}_{2}(b_1, a_1, b_2){\rho_2^+}^2e^{\rho_2^+(x-a)}-\hat{K}_{1}(b_1, a_1, b_2){\rho_1^+}^2e^{\rho_1^+(x-a)}}{w_{b_2}(b_2, a_1)},
\end{align}
where
\begin{align}
&\hat{K}_{1}(b_1, a_1, b_2):=g_{2,q}(a_1)-g_{2,q}'(b_2)V_{b_1,a_1, b_2}(a_1)
-c_{2}(q)\big(g_{1,q}(a_1)-g_{1,q}'(b_2)V_{b_1,a_1, b_2}(a_1)\big),\nonumber\\
&\hat{K}_{2}(b_1, a_1, b_2)
:=\big(1-c_{2}(q)\big)\big(g_{1,q}(a_1)-g_{1,q}'(b_2)V_{b_1,a_1, b_2}(a_1)\big).
\label{f11}
\end{align}
The proof for the next proposition is deferred to Appendix \ref{va1vb21kl34}.
\begin{prop}\label{va1vb21kl}
For any $0\leq b_1< a< b_2$ and $0\leq b_1\leq a_1\leq b_2$, we have
\[V_{b_1,a_1, b_2}'(a_1-)=1 \ \ \mathrm{and} \ \ V_{b_1,a_1, b_2}'(b_2)=1.\]
\end{prop}

Note that $V_{b_1,a_1, b_2}'(a_1-)=1$ by Proposition \ref{va1vb21kl}. Then the HJB inequality \eqref{hjb4} requires that $V_{b_1,a_1, b_2}'(a_1+)\leq1$, and \eqref{hjb2} further requires that  $V_{b_1,a_1, b_2}'(a_1+)=1$, leading to  $V_{b_1,a_1, b_2}'(a_1)=1$.
On the other hand, since $V_{b_1,a_1, b_2}'(x)=1$ for $x\in[b_2, \infty)$ by definition, inequality \eqref{hjb2} requires that
$V_{b_1,a_1, b_2}'(x)\geq1$ for $x\in(a, b_2)$, which implies that $b_2$ is a local minimum of $V_{b_1,a_1, b_2}'(x)$ for $x\in(a, \infty)$, leading to  $V_{b_1,a_1, b_2}''(b_2)=0$.
Therefore, to obtain the optimal band strategy we first consider $a_1$ and $b_2$ satisfying the following equations and then identify additional conditions for all the HJB inequalities to hold.
\begin{align}
& V_{b_1,a_1, b_2}'(a_1)
=1,\label{b11a1b2}\\
& V_{b_1,a_1, b_2}''(b_2)
=0.\label{asdfg}
	\end{align}
Also notice that if $b_2=a+$, then $V_{b_1,a_1, b_2}(x)$ is not differentiable at $x=a$ and \eqref{asdfg} is not relevant.

To analyze the monotonicity of $V_{b_1,a_1, b_2}'(x)$ for $x\in(a_1, \infty)$, we present three lemmas. Their proofs are deferred to Appendix \ref{b1a1225}, \ref{b1a12} and \ref{lem k12222}, respectively. Recall that, for $W''(x)$ given by \eqref{b11}, $W''(x)=0$ has a unique solution $b_-\in \mathbb{R}$ given in \eqref{b-}. By Lemma \ref{lem W} (I), for $0\leq x< a_1\leq a$, if $b_-\leq0$, then $W'(x)$ has a unique minimum at $0$, whereas if $0<b_-<a_1$, then $W'(x)$ has a unique minimum at $b_-$.
\begin{lem}\label{lem b1a125}
For fixed $b_1$ and $b_2$ satisfying $0\leq b_1<a< b_2$, if there exists $a_1\in[b_1, a)$ such that $V_{b_1,a_1, b_2}'(a_1)=1$, then we have
\begin{align}
&\tilde{K}_{1}(b_1, a_1, b_2)
=\frac{1-\rho_2^-V_{b_1, a_1, b_2}(a_1)}{\rho_2^- -\rho_1^-}
w_{b_2}(b_2, a_1) e^{-\rho_1^-(a_1-a)},\label{fde33216}\\
&\tilde{K}_{2}(b_1, a_1, b_2)
=\frac{1-\rho_1^-V_{b_1, a_1, b_2}(a_1)}{\rho_2^- -\rho_1^-}
w_{b_2}(b_2, a_1) e^{-\rho_2^-(a_1-a)},\label{fde3311}\\
&V_{b_1, a_1, b_2}''(a_1)=2\big(-\mu_-+qV_{b_1, a_1, b_2}(a_1)\big),\label{v2b12a1er}
\end{align}
where $w_{b_2}(b_2, a_1)$ is defined in \eqref{wb2b1fu1}.
\end{lem}
\begin{lem}\label{lem b1a1}
For fixed $b_1$ and $b_2$ satisfying $0\leq b_1<a< b_2$, where $b_1=b_-$ for $b_-\in(0,a)$ and $b_1=0$ for $b_-\leq 0$, if there exists $a_1\in[b_1, a)$ such that $V_{b_1,a_1, b_2}'(a_1)=1$, then $V_{b_1, a_1, b_2}'(x)$ is strictly increasing for $x\in(a_1, a)$.
\end{lem}
\begin{lem}\label{lem k12122}
For fixed $b_1$ and $a_1$ with $0\leq b_1 < a$ and $0\leq b_1\leq a_1$, if there exists $b_2>(a_1 \vee a)$ such that $V_{b_1,a_1, b_2}''(b_2)=0$, then $V_{b_1,a_1, b_2}'(x)$ is strictly decreasing for $x\in[a_1, b_{2})\cap (a, b_2)$.
\end{lem}

The remark below compares the expected present value of dividends for the barrier and band strategies, its proof is deferred to Appendix \ref{orderv1v200}. Let $\mathcal M$ be as in \eqref{mathcalmin}.
\begin{rem}\label{orderv1v2}
Let $b_1=b_-$ for $0<b_-<a$, $b_1=a-$ for $a\leq b_-$ and $b_1=0$ for $b_-\leq 0$. For $0\le b_1<a<b_2$, the following results hold.
\begin{enumerate}
\item[(i)] If $\mathcal M\in \{a+, b_+\}$, then for every $x\in[0,b_1)$ one has $V_{b_1,a_1,b_2}(x)\leq \max\{V_{a+}(x), V_{b_+}(x)\}$, hence no $(b_1,a_1,b_2)$-band is optimal.
    Optimality of $(b_1, a_1, b_2)$-band can occur only if $\mathcal M\in \{0+, b_-, a-\}$.

\item[(ii)] If $\mathcal M\in\{0+, b_-, a-\}$, and for some $a_1\in[b_1, a)$ with $V_{b_1,a_1, b_2}'(a_1)=1$, then for every $x\in(a_1, a)$ one has $V_{b_1,a_1,b_2}(x)\;> V_{b_1}(x)>\max\{V_{a+}(x),V_{b_+}(x)\}$, hence no barrier is optimal.

\item[(iii)] If $\mathcal M\in\{0+, b_-, a-\}$, and for some $b_2> a_1 \geq a$ with $V_{b_1,a_1,b_2}''(b_2)=0$, then for every $x\in(a_1, b_2)$ one has $V_{b_1,a_1,b_2}(x)> V_{b_1}(x)>\max\{V_{a+}(x),V_{b_+}(x)\}$, hence no barrier is optimal.
\end{enumerate}
\end{rem}
\subsubsection{The case  $b_1\leq a_1<a$.}\label{sec521}
\begin{thm}\label{thm vbb5}
Let $b_1=b_-$ for $0<b_-<a$ and $b_1=0$ for $b_-\leq 0$.
\begin{enumerate}[label=(\Roman*)]
\item (Optimality of $(b_1, a_1, b_2)$-band)
 If there exists  $(a_1, b_2)$ for $a_1\in [b_1, a)\cap(0,a)$ and $b_2\in (a,\infty)$ such that $V_{b_1,a_1, b_2}$ satisfies equations \eqref{b11a1b2}-\eqref{asdfg}, then $V_*=V_{b_1,a_1, b_2}$ and $\pi_*=\pi_{b_1,a_1, b_2}$.
\item (Optimality of $V_{b_1, a_1, a+}$)
If there exists  $a_1 \in [b_1, a)\cap(0,a)$ such that $V_{b_1, a_1, a+}$ satisfies equation \eqref{b11a1b2} and
$\mu_+-q V_{b_1, a_1, a+}(a)\leq 0$,
then $V_*=V_{b_1,a_1, a+}$.
\end{enumerate}
\end{thm}
\begin{proof}
By Lemma \ref{thm vbc2} and the definitions of $b_-$, $a_1$ and $b_2$, we get $V_{b_-,a_1, b_2}\in C(\mathbb{R_+})\cap C^2(\mathbb{R}_+\backslash\{a_1, a\})$ with $\mathcal{P}_{\pi_{b_-,a_1, b_2}}=\{a_1\}$. For $x\in[0, a_1)$, it follows from \eqref{val 1} and Lemma \ref{thm vbc2} that $V_{b_1}=V_{b_1,a_1, b_2}$, and the HJB inequalities \eqref{hjb1}-\eqref{hjb4} are verified for $b_1=0$ and $b_1=b_-$ in Theorems \ref{thm v0+} and \ref{thm vb-}, respectively. Notably, $V_{b_-, a_1, b_2}(b_-)=V_{b_-}(b_-)=\mu_-/{q}$. Here, we focus on the case $x\in[a_1, \infty)$, and consider $b_1 = b_-$ as an example, since the proof for $b_1 = 0$ follows similarly.

First, we prove case (I) beginning with  \eqref{hjb1}.
For $x\in[a_1, a)\cup (a, b_2)$, by \eqref{wxy}
\begin{align*}
&w(x, a_1)+w_{b_2}(b_2, x)V_{b_-,a_1, b_2}(a_1)\\
&=\Big(g_{1,q}(a_1)-g_{1,q}'(b_2)V_{b_-,a_1, b_2}(a_1)\Big)g_{2,q}(x)
-\Big(g_{2,q}(a_1)-g_{2,q}'(b_2)V_{b_-,a_1, b_2}(a_1)\Big)g_{1,q}(x),
\end{align*}
where $g_{i,q}(x), i=1,2$ are given by \eqref{g1} and \eqref{g2}, respectively. Then by \eqref{g1g2g3} we have
\begin{eqnarray*}
&&\frac{1}{2}V_{b_-,a_1, b_2}''(x)+(\mu_+\mathbf{1}_{\{x>a\}}
+\mu_-\mathbf{1}_{\{x< a\}})V_{b_-,a_1, b_2}'(x)-qV_{b_-,a_1, b_2}(x)=0.
\end{eqnarray*}
By the definition of $b_2$ we get $V_{b_-, a_1, b_2}''(b_2)=0$, and by Proposition \ref{va1vb21kl} we have $V_{b_-, a_1, b_2}'(b_2)=1$. Then $V_{b_-, a_1, b_2}(b_2)={\mu_+}/{q}$.
For $x\in[b_2, \infty)$, since $V_{b_-, a_1, b_2}'(x)=1$, $V_{b_-, a_1, b_2}''(x)=0$ and $V_{b_-, a_1, b_2}(x)\geq V_{b_-, a_1, b_2}(b_2)={\mu_+}/{q}$, we have
\begin{eqnarray*}
&&\frac{1}{2}V_{b_-, a_1, b_2}''(x)+\mu_+V_{b_-, a_1, b_2}'(x)-q V_{b_-, a_1, b_2}(x)
\leq \mu_+-q V_{b_-, a_1, b_2}(b_2)=0.
\end{eqnarray*}
Thus, \eqref{hjb1} holds.

We then prove \eqref{hjb2}. By the definition of $a_1$ we get $V_{b_-,a_1, b_2}'(a_1)=1$. For $x\in(a_1, a)$, by Lemma \ref{lem b1a1} we obtain that, $V_{b_-,a_1, b_2}'(x)$ is strictly increasing in $x$, and then $V_{b_-,a_1, b_2}'(x)> V_{b_-,a_1, b_2}'(a_1)=1$. For $x\in(a, b_2)$, by Proposition \ref{va1vb21kl} and Lemma \ref{lem k12122} we have $V_{b_-, a_1, b_2}'(x)> V_{b_-, a_1, b_2}'(b_2)=1$. For $x\in [b_2, \infty)$, $V_{b_-,a_1, b_2}'(x)=1$. Therefore, $1-V_{b_-,a_1, b_2}'(x)\leq0$ for $x\in\mathbb{R}_+\backslash\{a\}$, i.e. \eqref{hjb2} holds. We then verify \eqref{hjb3}. Since
\begin{align*}
V_{b_-, a_1, b_2}'(a\pm)={\big(w_{b_2}(b_2, a_1)\big)^{-1}} & \Big(\big(g_{1,q}(a_1)-g_{1,q}'(b_2)V_{b_-, a_1, b_2}(a_1)\big)g_{2,q}'(a\pm)\\
& -\big(g_{2,q}(a_1)-g_{2,q}'(b_2)V_{b_-, a_1, b_2}(a_1)\big)g_{1,q}'(a\pm)\Big),
\end{align*}
by \eqref{g1qx} we have $(1+\beta) V_{b_-, a_1, b_2}'(a+)-(1-\beta) V_{b_-, a_1, b_2}'(a-)=0$,
i.e. \eqref{hjb3} holds.
Finally, recalling that $V_{b_-, a_1, b_2}'(a_1-)=V_{b_-, a_1, b_2}'(a_1+)=1$, since $\mathcal{P}_{\pi_{b_-, a_1, b_2}}=\{a_1\}$, we have \eqref{hjb4}.

Now, we prove case (II). The proofs of \eqref{hjb1} and \eqref{hjb2} for $x\in[a_1, a)$ are omitted as they are similar to case (I). We proceed with the proof for $x>a$. Combing $\mu_+-q V_{b_-, a_1, a+}(a)\leq 0$, $V_{b_-, a_1, a+}'(x)=1$, $V_{b_-, a_1, a+}''(x)=0$ and $V_{b_-, a_1, a+}(x)> V_{b_-, a_1, a+}(a)$, we have
\begin{align*}
&\frac{1}{2}V_{b_-, a_1, a+}''(x)+\mu_+V_{b_-, a_1, a+}'(x)-q V_{b_-, a_1, a+}(x)
< \mu_+-q V_{b_-, a_1, a+}(a)\leq0.
\end{align*}
Thus, \eqref{hjb1} holds. By Lemma \ref{thm vbc2} we get $V_{b_-, a_1, a+}'(x)=1$ and \eqref{hjb2} holds. To show \eqref{hjb3}, by $V_{b_-, a_1, a+}'(a+)=1$ and $w_{a+, a+}(a+, a+):=w_{x, y}(x, y)|_{x=y=a+}=0$ we have
\begin{align*}
&(1+\beta) V_{b_-, a_1, a+}'(a+)-(1-\beta) V_{b_-, a_1, a+}'(a-)
=(1+\beta)\\
&-(1-\beta)\frac{1+\beta}{1-\beta}\frac{w_{a+}(a+, a_1)+w_{a+, a+}(a+, a+)V_{b_-, a_1, a+}(a_1)}{w_{a+}(a+, a_1)}=(1+\beta)-(1+\beta)=0,
\end{align*}
i.e. \eqref{hjb3} holds. Since $\mathcal{P}_{\pi_{b_-, a_1, a+}}=\{a_1\}$ and $V_{b_-, a_1, a+}'(a_1)=1$, we have \eqref{hjb4} holds.
\end{proof}
\subsubsection{The case  $b_1< a \leq a_1$.}
Let $a\leq a_1$ in Lemma \ref{thm vbc2}, we obtain the following lemma, with a detailed proof provided in Appendix \ref{thm vbbproof}.
\begin{lem}\label{thm vbb}
	For $0\leq b_1< a\leq a_1<b_2$, we have
\begin{eqnarray*}
		&V_{b_1, a_1, b_2}(x)=\left \{
		\begin{array}{ll}
			\frac{W(x)}{W'(b_1)} & \mathrm{for}\ x\in[0, b_1), \\
x-b_1+\frac{W(b_1)}{W'(b_1)} & \mathrm{for}\ x\in[b_1, a_1), \\
\frac{W_+(x-a_1)}{W_+'(b_2-a_1)}+
\big(a_1-b_1+\frac{W(b_1)}{W'(b_1)}\big)
\frac{W_+'(b_2-x)e^{(\rho_2^++\rho_1^+)(x-a_1)}}{W_+'(b_2-a_1)}   & \mathrm{for}\ x\in[a_1, b_2), \\
x-b_2+\frac{W_+(b_2-a_1)}{W_+'(b_2-a_1)}+\big(a_1-b_1+\frac{W(b_1)}{W'(b_1)}\big) \frac{W_+'(0)e^{(\rho_2^++\rho_1^+)(b_2-a_1)}}{W_+'(b_2-a_1)} & \mathrm{for}\ x\in[b_2, \infty),
		\end{array} \right.
	\end{eqnarray*}
where $W(x)$ is given by \eqref{sca 1} and
\begin{align}\label{w+xeft}
&W_+(x):=e^{\rho_2^+x}-e^{\rho_1^+x}.
\end{align}
\end{lem}
\begin{rem}\label{rem52}
Under the condition $V_{b_1, a_1, b_2}''(b_2)=0$ for $a<a_1<b_2$, it is necessary that $V_{b_1, a_1, b_2}'(a_1)=1$ for both \eqref{hjb2} and \eqref{hjb4} to hold, which fails to agree with  Lemma \ref{lem k12122}.
Therefore, the HJB inequalities \eqref{hjb1}-\eqref{hjb4} cannot be all satisfied for $a<a_1<b_2$,  and we only consider the case of $a=a_1<b_2$.
\end{rem}

To determine the values of $\beta\in(-1,1)$ for which the left and right derivatives of $V_{b_1,a, b_2}(x)$ at $x=a$ satisfy the HJB inequality \eqref{hjb3}, we provide the lemma below. Its proof is deferred to Appendix \ref{s11}.
\begin{lem}\label{lem s11}
If $V_{b_1,a, b_2}''(b_2)=0$ for $0\leq b_1<a<b_2$, then
\begin{align*}
& S(\beta):
=(1+\beta) V_{b_1, a, b_2}'(a+)-(1-\beta), \, \beta\in(-1,1),
\end{align*}
is strictly increasing.
Further, there exists a unique $\beta^*\in(-1,0)$ such that $S(\beta^*)=0$, and $S(\beta)\leq 0$ if and only if $\beta\in(-1, \beta^*]$ where
\begin{align}\label{al2}
\beta^*:&
=\frac{1-V_{b_1, a, b_{2}}'(a+)}{1+V_{b_1, a, b_{2}}'(a+)}.
\end{align}
\end{lem}
\begin{thm}\label{thm vabb51}
(Optimality of $(b_1, a, b_2)$-band)
Let $b_1=b_-$ for $0<b_-<a$, $b_1=a-$ for $a\leq b_-$ and $b_1=0$ for $b_-\leq 0$. If there exists  $b_2\in(a,\infty)$ such that $V_{b_1,a, b_2}$ satisfies equation \eqref{asdfg} and $\beta\in(-1, \beta^*]$, then $V_*=V_{b_1, a, b_2}$.
\end{thm}
\begin{proof}
By Lemma \ref{thm vbb} and the definitions of $b_1$ and $b_2$, we get $V_{b_1,a, b_2}\in C(\mathbb{R_+})\cap C^2(\mathbb{R}_+\backslash\{ a\})$ with $\mathcal{P}_{\pi_{b_1,a, b_2}}=\emptyset$. For $x\in[0, a)$, it follows from \eqref{val 1} and Lemma \ref{thm vbb} that $V_{b_1}=V_{b_1,a, b_2}$, and the HJB inequalities \eqref{hjb1}-\eqref{hjb4} are verified for $b_1=\{0, b_-, a-\}$ in Theorems \ref{thm v0+}-\ref{thm va-}. We now consider $x\in(a, \infty)$.

For the proof of \eqref{hjb1}, the cases where $b_1=b_-$ and $b_1=0$ are omitted since they follow similarly to Theorem \ref{thm vbb5}, and we consider the case $b_1 = a-$. For $x\in(a, b_2)$, since
\begin{align*}
&W_+(x-a)+\frac{W(a)}{W'(a-)}W_+'(b_2-x)e^{(\rho_2^++\rho_1^+)(x-a)}\\
&=\big(1-\frac{W(a)}{W'(a-)} \rho_1^+ e^{\rho_1^+(b_2-a)}\big) e^{\rho_2^+(x-a)}-
\big(1-\frac{W(a)}{W'(a-)} \rho_2^+ e^{\rho_2^+(b_2-a)}\big) e^{\rho_1^+(x-a)},
\end{align*}
by \eqref{dfgh212} we have
\begin{eqnarray*}
&&\frac{1}{2}V_{a-,a, b_2}''(x)+\mu_+ V_{a-,a, b_2}'(x)-qV_{a-,a, b_2}(x)=0.
\end{eqnarray*}
By the definition of $b_2$ we get $V_{a-, a, b_2}''(b_2)=0$, and by Proposition \ref{va1vb21kl} we have $V_{a-, a, b_2}'(b_2)=1$. Then $V_{a-, a, b_2}(b_2)={\mu_+}/{q}$.
For $x\in[b_2, \infty)$, since $V_{a-, a, b_2}'(x)=1$, $V_{a-, a, b_2}''(x)=0$ and $V_{a-, a, b_2}(x)\geq V_{a-, a, b_2}(b_2)={\mu_+}/{q}$, we have
\begin{eqnarray*}
&&\frac{1}{2}V_{a-, a, b_2}''(x)+\mu_+V_{a-, a, b_2}'(x)-q V_{a-, a, b_2}(x)
\leq \mu_+-q V_{a-, a, b_2}(b_2)=0.
\end{eqnarray*}
Thus, \eqref{hjb1} holds.

We next prove \eqref{hjb2}.
For $x\in(a, b_2)$, by Proposition \ref{va1vb21kl} and Lemma \ref{lem k12122} we have $V_{b_1, a, b_2}'(x)> V_{b_1, a, b_2}'(b_2)=1$. For $x\in[b_2, \infty)$, $V_{b_1, a, b_2}'(x)=1$. Thus, \eqref{hjb2} holds.
In addition, since $V_{b_1,a,b_2}'(a-)=1$, under the condition $\beta\in(-1, \beta^*]$, by Lemma \ref{lem s11} we have
$$(1+\beta)V_{b_1,a,b_2}'(a+)-(1-\beta)V_{b_1,a,b_2}'(a-)
=(1+\beta)V_{b_1,a,b_2}'(a+)-(1-\beta)\leq 0,$$
i.e. \eqref{hjb3} holds. Since $\mathcal{P}_{\pi_{b_1, a, b_2}}=\emptyset$, we have \eqref{hjb4}.
\end{proof}
Combining the analysis in the barrier and band sections  \ref{barrier_strategy}-\ref{band_strategy}, we obtain the following barrier-band optimality decision rule, with proof given in Appendix \ref{prop:decision00}. Let $b_1=b_-$ for $0<b_-<a$, $b_1=a-$ for $a\leq b_-$ and $b_1=0$ for $b_-\leq 0$, and $\mathcal{M}$ be as defined in \eqref{mathcalmin}. For brevity, write HJB for inequalities \eqref{hjb1}-\eqref{hjb4}.
\begin{rem}[Barrier vs.\ band: decision rule]\label{prop:decision}
With $\mathcal{M}\in\{a+,\,b_+\}$, the optimal policy is a barrier strategy: $b_+$ if $b_+>a$, and $a+$ if $b_+\leq a$.
With $\mathcal{M}\in\{0+,\,b_-,\,a-\}$, for $0\leq b_1\leq a_1, a\le b_2$ and $a_1>0$,  the optimal policy is determined by the following case-wise rules.
\begin{enumerate}[label=\textbf{(C\arabic*)}, leftmargin=2.4em, itemsep=0pt]
\item \textbf{Case \(a_1\in[b_1,\,a)\).}
  \begin{enumerate}[label=(\roman*), leftmargin=0.5em, labelsep=0.4em, itemsep=0pt, topsep=-1pt, partopsep=0pt]
  \item If \(V'_{b_1,a_1,b_2}(a_1)=1\), then the $b_1$-\emph{barrier} fails to be optimal, and further
    \begin{itemize}[label=\SmallDot, leftmargin=*, labelsep=0.4em, itemsep=0pt, topsep=-2pt, partopsep=0pt]
      \item If \(b_2>a\) and \(V''_{b_1,a_1,b_2}(b_2)=0\), the \((b_1,a_1,b_2)\)-\emph{band} is optimal.
      \item If \(b_2=a\) and \(\mu_+-q\,V_{b_1,a_1,a+}(a)\le 0\), the \((b_1,a_1,a+)\)-\emph{band} is optimal.
    \end{itemize}
  \item If \(V'_{b_1,a_1,b_2}(a_1)\neq1\), then the \((b_1,a_1,b_2)\)-\emph{band} fails to satisfy HJB, and further
    \begin{itemize}[label=\SmallDot, leftmargin=*, labelsep=0.4em, itemsep=0pt, topsep=-2pt, partopsep=0pt]
      \item If \(\beta\in(-1,0]\) and \(\mu_+ - q V_{b_1}(a)\leq 0\), the $b_1$\emph{-barrier} is optimal.
    \end{itemize}
  \end{enumerate}
\item \textbf{Case \(a_1=a\).}
  \begin{enumerate}[label=(\roman*), leftmargin=0.5em, labelsep=0.4em, itemsep=0pt, topsep=-2pt, partopsep=0pt]
  \item If \(b_2>a\), then
    \begin{itemize}[label=\SmallDot, leftmargin=*, labelsep=0.4em, itemsep=0pt, topsep=-2pt, partopsep=0pt]
      \item If \(V''_{b_1,a,b_2}(b_2)=0\) and \(\beta\in(-1,\beta^*]\), the \((b_1,a,b_2)\)\emph{-band} is optimal.
      \item If \(V''_{b_1,a,b_2}(b_2)\neq0\), the \((b_1,a,b_2)\)\emph{-band} fails to satisfy HJB, and further
          \begin{itemize}[label=\CDot, leftmargin=*, labelsep=0.4em, itemsep=0pt, topsep=-2pt, partopsep=0pt]
      \item If \(\beta\in(-1,0]\) and \(\mu_+ - q V_{b_1}(a)\leq 0\), the $b_1$\emph{-barrier} is optimal.
    \end{itemize}
    \end{itemize}
  \item If \(b_2=a\), then
    \begin{itemize}[label=\SmallDot, leftmargin=*, labelsep=0.4em, itemsep=0pt, topsep=-2pt, partopsep=0pt]
      \item If \(\beta\in(-1,0]\) and \(\mu_+ - q V_{b_1}(a)\leq 0\), the $b_1$\emph{-barrier} is optimal.
    \end{itemize}
  \end{enumerate}
  \item \textbf{Case \(a_1>a\).}
  The \((b_1,a_1,b_2)\)\emph{-band} fails to satisfy HJB.
\end{enumerate}
\end{rem}

\section{Examples}\label{example}
Applying Theorems \ref{thm v0+}-\ref{thm vb2} and Theorems \ref{thm vbb5}-\ref{thm vabb51},
we summarize in Table \ref{Table 1} and Table \ref{Table 2} the optimal strategies for different choices of (numerical) values for $\beta, \mu_-$ and $\mu_+$, respectively.
\begin{table}[!ht]
\caption{The optimal dividend strategies for $q=0.1$ and $a=1$.}
\label{Table 1}
\centering
\tabstyle
\begin{tabular}{c| c| c c c c c}
\toprule
$\beta$ & $\mu_-\backslash \mu_+$ & $-8$ & $1$ & $5$ & $7.1$ & $11$ \\
\multirow{5}*{$-0.9$} & $-5$ & $0$ & $(0,a,3.756)$ & $(0,0.684,2.232)$ & $(0,0.614,1.968)$ & $(0,0.529,1.703)$ \\

& $0$ & $0$ & $(0,a,3.756)$ & $b_+=2.199$ & $b_+=1.935$ & $b_+=1.674$  \\

& $1$ & $a-$ & $(a-,a,3.626)$ & $b_+=2.150$ & $b_+=1.895$ & $b_+=1.643$  \\

& $5$ & $a-$ & $a-$ & $(a-,a,1.718)$ & $b_+=1.715$ & $b_+=1.522$  \\

& $7$ & $b_-=0.982$ & $b_-=0.982$ & $b_-=0.982$ & $(0.982,a,1.605)$ & $b_+=1.506$  \\
\midrule
$\beta$ & $\mu_-\backslash \mu_+$ & $-8$ & $0.3$ & $0.5$ & $5$ & $10$ \\
\multirow{5}*{$-0.3$} & $-5$ & $0$ & $0$ & $(0,0.946,3.593)$ & $(0,0.488,2.195)$ & $(0,0.378,1.723)$  \\

& $0$ & $0$ & $(0,a, 2.753)$ & $b_+=3.456$ & $b_+=2.058$ & $b_+=1.639$  \\

& $1$ & $a-$ & $a-$ & $(a-,a,2.896)$ & $b_+=1.963$ & $b_+=1.589$  \\

& $7$ & $b_-=0.982$ & $b_-=0.982$ & $b_-=0.982$ & $b_-=0.982$ & $b_+=1.405$  \\

& $9$ & $b_-=0.820$ & $b_-=0.820$ & $b_-=0.820$ & $b_-=0.820$ & $b_+=1.383$  \\
\midrule

$\beta$ & $\mu_-\backslash \mu_+$ & $-8$ & $0$ & $6$ & $15$ & $50$ \\
\multirow{5}*{$0.3$}
& $-5$ & $(0,0.939,a+)$ & $(0,0.939,a+)$ & $(0,0.401,1.996)$ & $(0,0.287,1.498)$ & $(0,0.155,1.187)$  \\

& $-1$ & $(0,0.709,a+)$ & $(0,0.709,a+)$ & $b_+=1.898$ & $b_+=1.454$ & $b_+=1.173$  \\

& $0$ & $a+$ & $a+$ & $b_+=1.835$ & $b_+=1.428$ & $b_+=1.165$  \\

& $2$ & $a+$ & $a+$ & $b+=1.633$ & $b_+=1.352$ & $b_+=1.143$  \\

& $9$ & $a+$ & $a+$ & $a+$ & $a+$ & $b_+=1.115$  \\
\midrule

$\beta$ & $\mu_-\backslash \mu_+$ & $-8$ & $0$ & $1$ & $10$ & $171.2$ \\
\multirow{5}*{$0.9$} & $-5$ & $(0,0.707,a+)$ & $(0,0.707,a+)$ & $(0,0.583,2.963)$ & $(0,0.315,1.574)$ & $(0,0.027,1.059)$  \\

& $-1$ & $a+$ & $a+$ & $b_+=1.845$ & $b_+=1.501$ & $b_+=1.071$  \\

& $0$ & $a+$ & $a+$ & $a+$ & $b_+=1.456$ & $b_+=1.069$  \\

& $2$ & $a+$ & $a+$ & $a+$ & $a+$ & $b_+=1.063$  \\

& $9$ & $a+$ & $a+$ & $a+$ & $a+$ & $b_+=1.056$  \\
\bottomrule
\multicolumn{7}{c}{Note: Optimal strategies include types $0$, $b_{\pm}$, $a{\pm}$, $(0,a_1,b_2)$, $(0, a_1,a+)$, $(0, a, b_2)$,$(b_-, a, b_2)$ and $(a-, a, b_2)$.}
\end{tabular}
\end{table}

\begin{table}[!ht]
\caption{The optimal dividend strategies for $q=0.3$ and $a=2$.}
\label{Table 2}
\centering
\tabstyle
\begin{tabular}{c| c| c c c c c}
\toprule
$\beta$ & $\mu_-\backslash \mu_+$ & $-6$ & $2.538$ & $3.235$ & $3.705$ & $5.029$ \\

\multirow{5}*{$-0.9$} & $-3$ & $0$ & $(0,1.945,3.373)$ & $(0,1.856,3.247)$ & $(0,1.807,3.170)$ & $(0,1.698,2.992)$  \\

& $1.021$ & $b_-=1.700$ & $(1.700,1.743,3.3)$ & $b_+=3.160$ & $b_+=3.080$ & $b_+=2.903$  \\

& $1.941$ & $b_-=1.578$ & $(1.578,a,2.977)$ & $b_+=3.100$ & $b_+=3.024$ & $b_+=2.859$  \\

& $2.392$ & $b_-=1.468$ & $b_-=1.468$ & $(1.468,a,2.986)$ & $b_+=3.004$ & $b_+=2.843$   \\

& $3.971$ & $b_-=1.155$ & $b_-=1.155$ & $b_-=1.155$ & $b_-=1.155$ & $(1.155,1.412,2.801)$ \\
\midrule

$\beta$ & $\mu_-\backslash \mu_+$ & $-6$ & $1.131$ & $1.294$ & $1.421$ & $5.160$ \\
\multirow{5}*{$-0.3$}

& $-3$ & $0$ & $(0,1.947,3.257)$ & $(0,1.888,3.370)$ & $(0,1.846,3.415)$ & $(0,1.396,2.931)$\\

& $0.053$ & $b_-=0.176$ & $(0.176,1.601,3.230)$ & $(0.176,1.020,3.293)$ & $(0.176,0.234,3.300)$ & $b_+=2.791$ \\

& $0.499$ & $b_-=1.316$ & $(1.316,1.332,3.101)$ & $b_+=3.135$ & $b_+=3.147$ & $b_+=2.743$  \\

& $0.814$ & $b_-=1.632$ & $b_-=1.632$ & $(1.632,1.687,3.001)$ & $b_+=3.023$ & $b_+=2.710$  \\

& $4.692$ & $ b_-=1.052$ & $ b_-=1.052$ & $b_-=1.052$ & $b_-=1.052$ & $b_+=2.500$  \\
\midrule
$\beta$ & $\mu_-\backslash \mu_+$ & $-6$ & $1$ & $3$ & $6$ & $15$ \\
\multirow{5}*{$0.02498$}
& $-3$ & $(0,1.085,a+)$ & $(0,1.904,3.061)$ & $(0,1.497,3.186)$ & $(0,1.309,2.808)$ & $(0,1.109,2.425)$ \\

& $1$ & $a+$ & $a+$ & $b_+=2.746$ & $b_+=2.574$ & $b_+=2.327$ \\

& $3$ & $(1.332,1.363,a+)$ & $(1.332,1.363,a+)$ & $(1.332,1.363,a+)$ & $b_+=2.442$ & $b_+=2.287$ \\

& $5$ & $(1.013, 1.722, a+)$ & $(1.013, 1.722, a+)$ & $(1.013, 1.722, a+)$ & $b_+=2.290$ & $b_+=2.264$ \\

& $10$ & $a+$ & $a+$ & $a+$ & $a+$ & $b_+=2.215$ \\

\midrule
$\beta$ & $\mu_-\backslash \mu_+$ & $-6$ & $0.806$ & $0.844$ & $0.93677$ & $0.997$ \\
\multirow{5}*{$0.3$}

& $-3$ & $(0,1.910,a+)$ & $(0,1.828,2.991)$ & $(0,1.457,3.133)$ & $(0,1.283,2.774)$ & $(0,1.097,2.409)$ \\

& $0.041$ & $(0.136,0.997,a+)$ & $(0.136,0.918,2.273)$ & $(0.136,0.848,2.359)$ & $(0.136,0.578,2.522)$ & $b_+=2.600$  \\

& $0.061$ & $(0.202,0.893,a+)$ & $(0.202,0.810,2.251)$ & $(0.202,0.726,2.337)$ & $(0.202,0.202,2.500)$ & $b_+=2.577$  \\

& $0.076$ & $(0.252,0.787,a+)$ & $(0.252,0.697,2.231)$ & $(0.252,0.587,2.318)$ & $b_+=2.480$ & $b_+=2.559$  \\

& $0.095$ & $(0.314,0.558,a+)$ & $b_+=2.201$ & $b_+=2.287$ & $b_+=2.453$ & $b_+=2.535$  \\
\bottomrule
\multicolumn{7}{c}{Note: Optimal strategies include types $0$, $b_-$, $a+$, $b_+$, $(b_-,a_1,b_2)$, $(b_-,a_1,a+)$, $(0,a_1,b_2)$, $(0, a_1,a+)$ and $(b_-, a, b_2)$.}
\end{tabular}
\end{table}

From Tables \ref{Table 1}-\ref{Table 2} one can observe that, for $\beta\in(-1,0)$,
if $\mu_+<0$, then the optimal dividend strategy is the $b_1$-barrier strategy, where $b_1=b_-$ for $0<b_-<a$, $b_1=a-$ for $b_-\geq a$ and $b_1=0$ for $b_-\leq 0$, as mentioned in Theorem \ref{thm v0+} and Remarks \ref{rem vb-1} and \ref{rem va-1}. This suggests that a  constrained drift associated to the dynamics above $a>0$ makes it more challenging for surplus to reach high levels. Therefore, dividend should be paid before the surplus reached level $a$, and if $b_-\leq 0$, it is also meaningful to set the barrier at $0$ even through ruin immediately occurs as a result.

Furthermore, for fixed moderate $\beta\in(-1,0)$ and moderate $\mu_-$, as $\mu_+$ gradually increases from negative to positive, the optimal strategy undergoes a transition from a $b_1$-barrier strategy to a $(b_1, a_1,b_2)$-band strategy and ultimately to a $b_+$-barrier strategy. On the other hand, if $\mu_-$ takes an extreme negative value, then the $(0,a_1,b_2)$-band strategy tends to be optimal for large positive $\mu_+$ values, which agrees with the intuition that when the surplus is close to $0$, dividend should be paid as soon as possible in the presence of negative trend.

In addition, for $\beta\in(0, 1)$, if either $b_-\geq a$ or $0<b_-<a$ and $W'(b_-)> W'(a+)$, then the optimal strategy transitions from a $a+$-barrier strategy to a $b_+$-barrier strategy as $\mu_+$ increases; otherwise, it shifts from a $(b_1, a_1,a+)$-band strategy to a $(b_1, a_1,b_2)$-band strategy and then to a $b_+$-barrier strategy as $\mu_+$ increases.

Fixing $\beta\in(-1,0)$ and $\mu_-$, once the optimal strategy becomes a $b_+$-barrier, for sufficiently large values of $\mu_+>0$, the optimal barrier level $b_+$ exhibits a decreasing trend as $\mu_+$ increases. This is due to the fact that
to maximize the expected total amount of discounted dividends, the large value of $\mu+$ reduces the ruin probability and allows to set the barrier lower so that the dividend is paid earlier to reduce the effect of discounting.

To conclude, the tables suggest that either a barrier or a band strategy is optimal depending on the joint effect of $\beta$, $\mu_{\pm}$ and other model parameters. A band strategy may also be optimal when the drift switches its sign from negative to positive, i.e. $\mu_-< 0<\mu_+$ or when the skewness below $a$ intensifies, i.e. $\beta\downarrow -1$; in the latter case, a band strategy can remain optimal even for relatively small (or even zero) drift gap $|\mu_+ - \mu_-|$. For example, Table \ref{Table 1} reports that for $\beta= -0.9$ and $\mu_- = \mu_+ = 5$, the $(a-, a, 1.718)$-band is optimal.

\bmhead{Acknowledgements}
The first author thanks Concordia University where this work was completed during her visit. All authors thank the editor and the reviewers for constructive comments that significantly improved the paper.
This research is supported by Natural Sciences and Engineering Research Council of Canada (RGPIN-2021-04100), China Scholarship Council (No.202206840089), National Natural Science Foundation of China (Grant No.11671204) and Postgraduate Research \& Practice Innovation Program of Jiangsu Province (Project No.KYCX22\_0392).

\section*{Declarations}
Conflict of Interest: The authors declare that the presented results are new, and there is no any conflict of interest.

\begin{appendices}
\appendix

\section{Proofs}\label{Section 5}
In the following we provide proofs of lemmas and propositions for completeness.
Recall that $\rho_1^-<0<\rho_2^-, \rho_1^+<0<\rho_2^+$ and $\beta\in(-1,1)$.

\subsection{Proof of Lemma \ref{c1c201}}
\begin{proof}\label{c11c223}
From the definition of $c_i(q)$ for $i=1,2$ in \eqref{c1c266}, it follows that
\begin{align}
&1-c_1(q)=\frac{(1-\beta)\rho_2^--(1+\beta)\rho_1^+}{(1-\beta)(\rho_2^--\rho_1^-)},\label{1c1q1}\\
&1-c_2(q)=\frac{(1-\beta)\rho_2^--(1+\beta)\rho_1^+}{(1+\beta)(\rho_2^+-\rho_1^+)}.\nonumber
\end{align}
Since both the numerator and denominator of $1-c_i(q)$ for $i=1,2$ are positive, we have
\begin{align}\label{1c1q11}
&1-c_i(q)>0.
\end{align}
Further, since $(1\pm \beta)(\rho_{2}^{\pm}-\rho_{1}^{\pm})>0$ holds for the denominators of $c_i(q)$ for $i=1,2$, by \eqref{c1c266} it is straightforward to obtain the necessary and sufficient conditions for $c_i(q)<0$, as outlined in Lemma \ref{c1c201}.
\end{proof}
\subsection{Proof of Lemma \ref{value function}}
\begin{proof}\label{prof value}
For $0\leq a_0<b$ and $b\in\mathbb{R}_+\backslash\{a\}$, we consider three different scenarios: $a_0<a< b$, $a\leq a_0<b$ and $a_0<b< a$. For the case $a_0<a< b$, applying the strong Markov property, we have
\begin{align}\label{val 7}
V_{b}(b, a_0)
&= \mathbb{E}_{b}\big[\int_0^{\hat{\tau}_a} e^{-q t}\mathrm{d}D_t^{\pi_b}\big]
+\mathbb{E}_b[e^{-q \hat{\tau}_a}] \mathbb{E}_a[e^{-q \tau_b}; \tau_b<\tau_{a_0}] V_{b}(b, a_0).
	\end{align}
Considering that the process $(U_t^{\pi_b})_{0\leq t\leq \hat{\tau}_a}$ is simply a reflected Brownian motion with drift $\mu_+$, i.e. $U_t^{\pi_b}=x+B_t+\mu_+t-D_t^{\pi_b}$, by Proposition 1 in \cite{Renaud2007}, we have
\begin{align}\label{val 6}
\mathbb{E}_{b}\big[\int_0^{\hat{\tau}_a} e^{-q t}\mathrm{d}D_t^{\pi_b}\big]
=\frac{e^{\rho_2^+(b-a)}-e^{\rho_1^+(b-a)}}
{\rho_2^+e^{\rho_2^+(b-a)}-\rho_1^+e^{\rho_1^+(b-a)}}
=\frac{w(b,a)}{w_b(b,a)},
\end{align}
where $w_{x}(x,y)$ is given by \eqref{wxpart}.

We now determine $\mathbb{E}_b[e^{-q \hat{\tau}_a}]$ with $a<b$. Fix $x\in(a, b]$ and consider $X$ on $[0, \hat{\tau}_a]$. Define the running maximum process by $\overline{X}_t:=\max\limits_{0\leq s\leq t}X_s,\, t\geq 0$. Then the aggregate dividends paid up to time $t$ is
$D_t^{\pi_b}:=(\overline{X}_t-b)\vee 0$,
with $D_{0}^{\pi_b}=0$.
Let $S_t:=0\vee \overline{X}_t$, define $Y_t:=S_t-X_t$ as the Brownian motion $X_t$ with drift $\mu_+$ reflected at its past running maximum $S_t$. Denote by $\tilde{\tau}_y$ the time at which process $(Y_t)_{t\geq 0}$ first hit the boundary $y$, i.e. $\tilde{\tau}_{y}:=\inf\{t\geq 0, Y_t\geq y\}$. By construction, $U_t^{\pi_b} = X_t - D_t^{\pi_b}$ is simply a reflected Brownian motion with drift $\mu_+$, i.e. $U_t^{\pi_b}=x+B_t+\mu_+t-D_t^{\pi_b}$.
Then applying the spatial homogeneity of $X$, it is easy to conclude that $\{U^{\pi_b}, D^{\pi_b}, \hat{\tau}_a, U_0^{\pi_b}=x\}$ has the same law as $\{b-Y, S, \tilde{\tau}_{b-a}, Y_0=b-x\}$. For any $q>0$, by Theorem 4.1 and Remark 4.3 of \cite{Zhou2007} we have
\begin{align*}
\mathbb{E}_{U_0^{\pi_b}=b}[e^{-q \hat{\tau}_a}]=\mathbb{E}_{Y_0=0}[e^{-q \tilde{\tau}_{b-a}}]=\frac{K^{(q)}(b-a)}
{\widehat{W}^{(q)}{'}(b-a)},
\end{align*}
where $\widehat W^{(q)}$ denotes the $q$-scale function of a Brownian motion with
drift $\mu_+$, i.e. $\widehat W^{(q)}(x) = \frac{2}{\rho_2^+ - \rho_1^+}(e^{\rho_2^+ x} - e^{\rho_1^+ x})$, and
\begin{align*}
&K^{(q)}(x):= q \big(\int_0^x \widehat{W}^{(q)}(y) \mathrm{d} y\big) \, \widehat{W}^{(q)}{'}(x)
-q \,\big(\widehat{W}^{(q)}(x)\big)^2 + \widehat{W}^{(q)}{'}(x).
\end{align*}
Since $\rho_1^+ \rho_2^+ = -2q$, further simplification yields
\begin{align}\label{val 5}
\mathbb{E}_{b}[e^{-q \hat{\tau}_a}]
&=\frac{1}{\rho_2^+ e^{\rho_2^+ (b-a)} - \rho_1^+ e^{\rho_1^+ (b-a)}} \nonumber \\
& \cdot \Big( \frac{2q}{\rho_2^+ - \rho_1^+}
\big( (\frac{e^{\rho_2^+ (b-a)} -1}{\rho_2^+} - \frac{e^{\rho_1^+ (b-a)} -1}{\rho_1^+})
(\rho_2^+ e^{\rho_2^+ (b-a)} - \rho_1^+ e^{\rho_1^+ (b - a) }) \nonumber \\
& \ \ \ \ \ \ \ \ \ \ \ \ \ \ \ \ \ \ \ \ \  - (e^{\rho_2^+ (b-a)} - e^{\rho_1^+(b-a)})^2 \big) + \rho_2^+ e^{\rho_2^+ (b-a)} - \rho_1^+ e^{\rho_1^+ (b-a)} \Big) \nonumber \\
& =\frac{(\rho_2^+-\rho_1^+) e^{(\rho_1^+ + \rho_2^+)(b-a)}}
{\rho_2^+e^{\rho_2^+(b-a)}-\rho_1^+e^{\rho_1^+(b-a)}}
=\frac{\rho_2^+-\rho_1^+}
{\rho_2^+e^{\rho_1^+(a-b)}-\rho_1^+e^{\rho_2^+(a-b)}}
=\frac{w_b(b,b)}
{w_b(b,a)},
\end{align}
where $w_{x}(x, x)=w_{y}(y, x)|_{y=x}$.

Thus, substituting \eqref{val 6} and \eqref{val 5} into \eqref{val 7}, by \eqref{e1} we obtain
\begin{align*}
&V_{b}(b, a_0)
=\frac{w(b,a)w(b,a_0)}{w_b(b,a)w(b, a_0)-w_b(b,b)w(a, a_0)}.
\end{align*}
By \eqref{wxy},
\begin{align}\label{w1w2w3rela}
&w_b(b,a)w(b, a_0)-w_b(b,b)w(a, a_0) \nonumber\\
&=\big(g_{2,q}'(b)g_{1,q}(a)-g_{1,q}'(b)g_{2,q}(a)\big)
\big(g_{2,q}(b)g_{1,q}(a_0)-g_{1,q}(b)g_{2,q}(a_0)\big) \nonumber \\
&\ \ -\big(g_{2,q}'(b)g_{1,q}(b)-g_{1,q}'(b)g_{2,q}(b)\big)
\big(g_{2,q}(a)g_{1,q}(a_0)-g_{1,q}(a)g_{2,q}(a_0)\big) \nonumber \\
&=g_{2,q}(b)g_{1,q}(a)\big(g_{2,q}'(b) g_{1,q}(a_0) - g_{1,q}'(b) g_{2,q}(a_0)\big) \nonumber \\
&\ \ - g_{1,q}(b) g_{2,q}(a) \big(g_{2,q}'(b)g_{1,q}(a_0) - g_{1,q}'(b) g_{2,q}(a_0)\big) \nonumber \\
&= \big(g_{2,q}(b)g_{1,q}(a) - g_{1,q}(b) g_{2,q}(a)\big)
\big(g_{2,q}'(b)g_{1,q}(a_0) - g_{1,q}'(b) g_{2,q}(a_0)\big) \nonumber \\
&= w(b,a) w_b(b, a_0),
\end{align}
which simplifies $V_{b}(b, a_0)$ to
\begin{align*}
&V_{b}(b, a_0)
=\frac{w(b,a)w(b,a_0)}{w(b, a)w_b(b, a_0)}
=\frac{w(b,a_0)}
{w_b(b,a_0)}.
\end{align*}
Note that for  $a\leq a_0<b$ and $a_0<b< a$, $U^{\pi_b}$ also reduces to a reflected Brownian motion with drift.  By Proposition 1 in \cite{Renaud2007} again, we have $V_{b}(b, a_0) = \frac{w(b,a_0)}{w_b(b,a_0)}$. Further, when $a_0=0$, by \eqref{vpib1} and \eqref{sca 1} we have
\begin{align}\label{vpib2}
&V_{b}(b)=\frac{W(b)}{W'(b)}.
\end{align}

For $0\leq x\leq b$, applying the strong Markov property together with the fact that no dividends are paid out until the surplus process $X$ exceeds the level $b$, by \eqref{e1}, \eqref{sca 1} and \eqref{vpib2} we have
\begin{eqnarray*}\label{val 3}
		&&V_b(x)=\mathbb{E}_x[e^{-q \tau_{b}}; \tau_{b}<\tau_0]V_b(b)
=\frac{W(x)}{W(b)} V_b(b)=\frac{W(x)}{W'(b)}.
	\end{eqnarray*}
For $x> b> a>0$, since $D^{\pi_b}$ has a jump at $t=0$ of size $x-b$ to bring $U^{\pi_b}$ back to the level $b$, by \eqref{vpib2} we have
	\begin{eqnarray*}\label{val 4}
		V_b(x)=x-b+V_b(b)=x-b+\frac{W(b)}{W'(b)}.
	\end{eqnarray*}
This completes the proof.
\end{proof}
\subsection{Proof of Proposition \ref{prop ww}}
\begin{proof}\label{prop ww11}
The first-order derivative of $W(x)$ defined by \eqref{sca 1} is given by
\[ W'(x) = g_{2,q}'(x)g_{1,q}(0)-g_{1,q}'(x)g_{2,q}(0). \]
From the forms of $g_{1,q}(x)$ and $g_{2,q}(x)$ in \eqref{g1} and \eqref{g2}, it follows that $g_{1,q}(0) = c_1(q)e^{-\rho_2^- a} + \big(1-c_1(q)\big) e^{-\rho_1^- a}$, $g_{2,q}(0) = e^{-\rho_2^- a}$, and
\begin{align*}
&g_{1,q}'(x)= \rho_{1}^+ e^{\rho_{1}^+(x-a)} {\mathbf{1}_{\{x> a\}}}+{\Big(c_1(q)\rho_{2}^-e^{\rho_{2}^-{(x-a)}}+\big(1-c_1(q)\big) \rho_{1}^- e^{\rho_{1}^-{(x-a)}} \Big)}{\mathbf{1}_{\{x\leq a\}}},\\
&g_{2,q}'(x)=\Big(\big(1-c_2(q)\big)\rho_{2}^+e^{\rho_{2}^+(x-a)}+c_2(q)\rho_{1}^+ e^{\rho_{1}^+(x-a)} \Big){1_{\{x> a\}}}{+\rho_{2}^- e^{\rho_{2}^-(x-a)}}{1_{\{x\leq a\}}}.
\end{align*}
Substituting $g_{i,q}'(x)$ for $i=1,2$ into $W'(x)$, and incorporating the given $g_{i,q}(0)$, we obtain, for $0\leq x< a$,
\begin{align*}
W'(x)& = \rho_{2}^- e^{\rho_{2}^-(x-a)} \Big(c_1(q)e^{-\rho_2^- a} + \big(1-c_1(q)\big) e^{-\rho_1^- a}\Big) \\
&\ \  - \Big(c_1(q)\rho_{2}^-e^{\rho_{2}^-{(x-a)}}+\big(1-c_1(q)\big) \rho_{1}^- e^{\rho_{1}^-{(x-a)}}\Big) e^{-\rho_2^- a}\\
&=
\big(1-c_1(q)\big)e^{-(\rho_1^-+\rho_2^-)a}
(\rho_2^-e^{\rho_2^-x}-\rho_1^-e^{\rho_1^-x}),
\end{align*}
and for $x>a$,
\begin{align*}
W'(x)&=\Big(\big(1-c_2(q)\big)\rho_{2}^+e^{\rho_{2}^+(x-a)}+c_2(q)\rho_{1}^+ e^{\rho_{1}^+(x-a)} \Big) \Big(c_1(q)e^{-\rho_2^- a} + \big(1-c_1(q)\big) e^{-\rho_1^- a}\Big)\\
&\ \
- \rho_{1}^+ e^{\rho_{1}^+(x-a)} e^{-\rho_2^- a}\\
&=\rho_2^+\Big(c_1(q)e^{-\rho_2^-a}+\big(1-c_1(q)\big)e^{-\rho_1^-a}\Big)
\big(1-c_2(q)\big)e^{\rho_2^+(x-a)}\\
&\ \  -\rho_1^+\Big(\big(1-c_1(q)c_2(q)\big)e^{-\rho_2^-a}-c_2(q)\big(1-c_1(q)\big)e^{-\rho_1^-a}\Big)
e^{\rho_1^+(x-a)}.
\end{align*}
In particular, by \eqref{g1qx} we have
\begin{align*}
(1-\beta)W'(a-)&=(1-\beta)\big(g_{2,q}'(a-)g_{1,q}(0)-g_{1,q}'(a-)g_{2,q}(0)\big)\\
&=(1+\beta)\big(g_{2,q}'(a+)g_{1,q}(0)-g_{1,q}'(a+)g_{2,q}(0)\big)\\
&=(1+\beta) W'(a+).
\end{align*}
This completes the proof.
\end{proof}
\subsection{Proof of Proposition \ref{propb-b+K}}
\begin{proof}\label{propb-b+K11}
The condition $W''(x)=0, x\in\mathbb{R}$, with $W''(x)$ given by \eqref{b11}, implies
\begin{align*}
\big(1-c_1(q)\big)e^{-(\rho_1^-+\rho_2^-)a}\big({\rho_2^-}^2e^{\rho_2^-x}
-{\rho_1^-}^2e^{\rho_1^-x}\big)=0,
\end{align*}
which simplifies to $e^{(\rho_2^- - \rho_1^-)x}
={\rho_1^-}^2/{\rho_2^-}^2$. Since ${\rho_1^-}^2/{\rho_2^-}^2>0$, solving for $x$ yields  $b_-$  in \eqref{b-}.
The condition $W''(x)=0, x\in\mathbb{R}$ implies by \eqref{b22} that
\begin{align*}
&{\rho_2^+}^2\big(1-c_2(q)\big)\Big(c_1(q)e^{-\rho_2^-a}+\big(1-c_1(q)\big)e^{-\rho_1^-a}\Big)
e^{\rho_2^+(x-a)}\\
& -{\rho_1^+}^2 \Big(\big(1-c_1(q)c_2(q)\big)e^{-\rho_2^-a}-c_2(q)\big(1-c_1(q)\big)e^{-\rho_1^-a}\Big)
e^{\rho_1^+(x-a)}=0.
\end{align*}
It follows that $e^{(\rho_2^+-\rho_1^+)(x-a)}=K(\beta)$ for $K(\beta)$ given by \eqref{K}. If $K(\beta)>0$, the expression for $b_+$ in \eqref{b+} is obtained.
\end{proof}
\subsection{Proof of Proposition \ref{b1min}}
\begin{proof}\label{b11min1}
\begin{enumerate}[label=(\roman*)]
\item If $\mu_-\leq 0$, then \eqref{dfgh21245} leads to $-\rho_1^- \leq \rho_2^-$, which in turn implies $\ln(-\rho_1^- / \rho_2^-)\leq 0$, and consequently, $b_-$, as defined in \eqref{b-}, satisfies $b_-\leq 0$.
Conversely, if $b_-\leq 0$, then by \eqref{b-} we get $\ln(-\rho_1^- / \rho_2^-)\leq0$, which further implies $-\rho_1^-\leq \rho_2^-$, and then, \eqref{dfgh21245} yield $\mu_-\leq 0$.
\item If $K(\beta)>1$, then as $\ln K(\beta)>0$, the expression of $b_+$ in \eqref{b+} implies $b_+> a$. Conversely, by \eqref{b+}, the necessary condition for $b_+> a$ is $\ln K(\beta)>0$, which further corresponds to $K(\beta)>1$.
\end{enumerate}
This completes the proof.
\end{proof}
\subsection{Proof of Lemma \ref{lem W}}
\begin{proof}\label{a.2}
Note that
\begin{align}\label{c1qea}
& c_1(q) e^{-\rho_2^-a}+\big(1-c_1(q)\big)e^{-\rho_1^- a}>c_1(q) e^{-\rho_2^-a}+\big(1-c_1(q)\big)e^{-\rho_2^- a}=e^{-\rho_2^-a}>0.
\end{align}
We first prove that $W(x)$ is a non-negative, continuous increasing function on $\mathbb{R}_+$ by considering three cases: $0\leq x< a$, $x=a$ and $x>a>0$.

For $0\leq x< a$,  by \eqref{wa-} we have $W'(x)>0$.

For $x=a$, by \eqref{wa-} we have
\begin{align*}
&W'(a-)=\big(1-c_1(q)\big)(\rho_2^-e^{-\rho_1^-a}-\rho_1^-e^{-\rho_2^-a})>0,
\end{align*}
and then, by \eqref{wa-a+} we have $W'(a+)>0$.

For $x>a>0$, the cases $c_2(q)< 0$ and $0\leq c_2(q)< 1$ are considered separately.
\begin{enumerate}[label=(\roman*)]
\item If $c_2(q)< 0$, then by \eqref{c1qea}
\begin{align*}
\rho_2^+ \Big(c_1(q)e^{-\rho_2^-a}+\big(1-c_1(q)\big)e^{-\rho_1^-a}\Big)
\big(1-c_2(q)\big)e^{\rho_2^+(x-a)}
>\rho_2^+ e^{-\rho_2^-a}
\big(1-c_2(q)\big)e^{\rho_2^+(x-a)},
\end{align*}
and
\begin{align*}
\rho_1^+ c_2(q)\big(1-c_1(q)\big)e^{-\rho_1^-a}
>\rho_1^+ c_2(q)\big(1-c_1(q)\big)e^{-\rho_2^-a}.
\end{align*}
Combining the above two inequalities,
we obtain
\begin{align*}
W'(x)
& >
\rho_2^+ e^{-\rho_2^-a}
\big(1-c_2(q)\big)e^{\rho_2^+(x-a)}\\
 & \ \ -\rho_1^+\Big(\big(1-c_1(q)c_2(q)\big)e^{-\rho_2^-a}-c_2(q)\big(1-c_1(q)\big)e^{-\rho_2^- a}\Big)
e^{\rho_1^+(x-a)}\\
& =
\big(1-c_2(q)\big)e^{-\rho_2^-a}
\big(\rho_2^+e^{\rho_2^+(x-a)}-\rho_1^+e^{\rho_1^+(x-a)}\big)>0.
\end{align*}
\item If $0\leq c_2(q)< 1$, then
$$\rho_2^+\big(1-c_2(q)\big)e^{\rho_2^+(x-a)}+\rho_1^+c_2(q)e^{\rho_1^+(x-a)}
>\rho_1^+c_2(q)\big(e^{\rho_1^+(x-a)}-e^{\rho_2^+(x-a)}\big)>0,$$
since
\begin{align}\label{ca66}
&\rho_2^+\big(1-c_2(q)\big)+\rho_1^+c_2(q)
=\rho_2^+ \frac{(1-\beta)\rho_2^--(1+\beta)\rho_1^+}{(1+\beta)(\rho_2^+-\rho_1^+)}
+\rho_1^+ \frac{(1+\beta) \rho_{2}^+-(1-\beta)\rho_{2}^-}{(1+\beta)(\rho_{2}^+-\rho_{1}^+)}\nonumber\\
&=\frac{(1-\beta)\rho_2^-\rho_{2}^+-(1-\beta)\rho_{2}^-\rho_1^+}
{(1+\beta)(\rho_{2}^+-\rho_{1}^+)}
=\frac{(1-\beta)\rho_2^-}{(1+\beta)}>0.
\end{align}
It follows from  \eqref{wa221} and \eqref{c1qea} that
\begin{align*}
W'(x)
&=\Big(c_1(q)e^{-\rho_2^-a}+\big(1-c_1(q)\big)e^{-\rho_1^-a}\Big)
\Big(\rho_2^+\big(1-c_2(q)\big)e^{\rho_2^+(x-a)}+\rho_1^+c_2(q)e^{\rho_1^+(x-a)}\Big)\\
& \ \ - \rho_1^+ e^{-(\rho_1^+ +\rho_2^-)a+\rho_1^+x}\\
& > e^{-\rho_2^-a}
\Big(\rho_2^+\big(1-c_2(q)\big)e^{\rho_2^+(x-a)}+\rho_1^+c_2(q)e^{\rho_1^+(x-a)}\Big)
- \rho_1^+e^{-(\rho_1^+ +\rho_2^-)a+\rho_1^+x}\\
&>0.
\end{align*}
\end{enumerate}

By \eqref{g1} and \eqref{g2} we have $g_{1,q}(a\pm)=g_{2,q}(a\pm)=1$, and then
\begin{eqnarray*}
&&W(a)=W(a-)=W(a+)=g_{1,q}(0)-g_{2,q}(0)
=\big(1-c_1(q)\big)\big(e^{-\rho_1^-a}-e^{-\rho_2^-a}\big)>0.
\end{eqnarray*}
In addition, $W(0) = g_{2,q}(0)g_{1,q}(0)-g_{1,q}(0)g_{2,q}(0)=0$. Combining the above,
$W$ is a non-negative, continuous increasing function on $\mathbb{R}_+$.

We next discuss the monotonicity and convexity of $W'(x)$, which can be divided into two cases: $0\leq x< a$ and $x>a>0$.
(\uppercase\expandafter{\romannumeral1})
For $0\leq x<a$, taking the derivative of $W''$ defined in \eqref{b11} yields
\begin{eqnarray*}
&&W'''(x)=\big(1-c_1(q)\big)e^{-(\rho_2^-+\rho_1^-)a}\big({\rho_2^-}^3e^{\rho_2^-x}
-{\rho_1^-}^3e^{\rho_1^-x}\big)>0.
\end{eqnarray*}
Then $W''(x)$ is strictly increasing in $x$.
Recall the solution $b_-$ in \eqref{b-} of  $W''(x)=0$.
In the following we  consider  different values of   $b_-$.
\begin{enumerate}[label=(\roman*)]
\item $b_-\leq 0$ if and only if $W''(x)>0$  if and only if  $W'(x)$ is strictly increasing in $x$,
\item $0<b_-<a$ if and only if $W''(x)<0$ for $x\in[0,b_-)$ and $W''(x)>0$ for $x\in(b_-, a)$ if and only if $W'(x)$ is non-monotone convex,
\item $a\leq b_-$ if and only if $W''(x)<0$ if and only if $W'(x)$ is strictly decreasing in $x$.
\end{enumerate}

(\uppercase\expandafter{\romannumeral2}) For $x>a>0$, we further discuss the monotonicity and convexity of $W'(x)$ for $K(\beta)\leq 1$ and $K(\beta)> 1$, respectively.
\begin{enumerate}[label=(\roman*)]
\item If $K(\beta)\leq 1$, then the numerator and denominator in \eqref{K} satisfy
\begin{align*}
&{\rho_1^+}^2\big((1-c_1(q)c_2(q))e^{-\rho_2^-a}-c_2(q)(1-c_1(q))
e^{-\rho_1^-a}\big)\\
& \leq {\rho_2^+}^2(1-c_2(q))(c_1(q)e^{-\rho_2^-a}+(1-c_1(q))e^{-\rho_1^-a}),
\end{align*}
which leads to $W''(x)$, defined by \eqref{b22}, satisfying
\begin{align*}
W{''}(x)
&\geq {\rho_2^+}^2\big(1-c_2(q)\big)\Big(c_1(q)e^{-\rho_2^-a}+\big(1-c_1(q)\big)e^{-\rho_1^-a}\Big)
e^{\rho_2^+(x-a)}\\
&\ \  -{\rho_2^+}^2\big(1-c_2(q)\big)\Big(c_1(q)e^{-\rho_2^-a}+\big(1-c_1(q)\big)e^{-\rho_1^-a}\Big)
e^{\rho_1^+(x-a)}\\
&= {\rho_2^+}^2\big(1-c_2(q)\big)\Big(c_1(q)e^{-\rho_2^-a}+\big(1-c_1(q)\big)e^{-\rho_1^-a}\Big)
(e^{\rho_2^+(x-a)}-e^{\rho_1^+(x-a)})\\
&> {\rho_2^+}^2 \big(1-c_2(q)\big) e^{-\rho_2^-a} (e^{\rho_2^+(x-a)}-e^{\rho_1^+(x-a)})>0,
\end{align*}
where \eqref{c1qea} is applied for the last inequality.
Then $W'(x)$ is strictly increasing.
\item
By \eqref{b22},
\begin{align*} W{'''}(x)&={\rho_2^+}^3\Big(c_1(q)e^{-\rho_2^-a}+\big(1-c_1(q)\big)e^{-\rho_1^-a}\Big)
	\big(1-c_2(q)\big)e^{\rho_2^+(x-a)}\\
	& \ \  -{\rho_1^+}^3\Big(\big(1-c_1(q)c_2(q)\big)e^{-\rho_2^-a}-c_2(q)\big(1-c_1(q)\big)
	e^{-\rho_1^-a}\Big)e^{\rho_1^+(x-a)}.
\end{align*}
If $K(\beta)>1$, since the denominator of $K(\beta)$ is positive, we have
\begin{align*}
&\big(1-c_1(q)c_2(q)\big)e^{-\rho_2^-a}-c_2(q)\big(1-c_1(q)\big)
e^{-\rho_1^-a}>0,
\end{align*}
which combined with  \eqref{c1qea} implies $W{'''}(x)>0$, i.e. $W''$ is strictly increasing.
By Proposition \ref{b1min} (ii) we have $b_+>a$ for $W{''}(b_+)=0$. Then $W''(x)<0$ for $x\in(a, b_+)$ and $W''(x)>0$ for $x\in(b_+, \infty)$. Thus, $W'$ is non-monotone convex on $(a, \infty)$.
\end{enumerate}
This completes the proof.
\end{proof}
\subsection{Proof of Proposition \ref{prop W1}}
\begin{proof}\label{a.4}
We write $c_{1,\beta}(q)$ for $c_1(q)$ to emphasize its dependence on $\beta$. Then
\begin{align}\label{sc1c21}
&\frac{\mathrm{d} c_{1,\beta}(q)}{\mathrm{d}\beta}
=\frac{\mathrm{d}}{\mathrm{d}\beta}\Big(\frac{(1+\beta) \rho_{1}^+-(1-\beta)\rho_{1}^-}
{(1-\beta)(\rho_{2}^--\rho_{1}^-)}\Big)\nonumber\\
&=\frac{(\rho_{1}^+ + \rho_{1}^-) (1-\beta)(\rho_{2}^--\rho_{1}^-) + \big((1+\beta)\rho_{1}^+ - (1-\beta)\rho_{1}^-\big)(\rho_{2}^--\rho_{1}^-)}
{(1-\beta)^2(\rho_{2}^--\rho_{1}^-)^2}\nonumber\\
&= \frac{(\rho_{1}^+ + \rho_{1}^-) (1-\beta) + (1+\beta)\rho_{1}^+ - (1-\beta)\rho_{1}^-}
{(1-\beta)^2(\rho_{2}^--\rho_{1}^-)}
=\frac{2\rho_1^+}{(1-\beta)^2(\rho_2^--\rho_1^-)}.
\end{align}

By \eqref{wa-},
\begin{align}\label{wa-betadre11}
&W_{\beta}'(a-)=\big(1-c_{1,\beta}(q)\big)(\rho_2^-e^{-\rho_1^-a}-\rho_1^-e^{-\rho_2^-a}).
\end{align}
Then applying \eqref{sc1c21} we have
\begin{align}\label{wa-betadre}
\frac{\mathrm{d}W_{\beta}'(a-)}{\mathrm{d}\beta}
&=\frac{-2\rho_1^+}{(1-\beta)^2(\rho_2^--\rho_1^-)}
(\rho_2^-e^{-\rho_1^-a}-\rho_1^-e^{-\rho_2^-a})>0,
\end{align}
and $W_{\beta}'(a-)$ is thus strictly increasing in $\beta$ for $\beta\in(-1, 1)$.
By (\ref{wa-a+}) we have
$W_{\beta}'(a+)=\frac{1-\beta}{1+\beta} W_{\beta}'(a-)$.
Then
\begin{align*}
&\frac{\mathrm{d}W_{\beta}'(a+)}{\mathrm{d}\beta}
=\frac{\mathrm{d}}{\mathrm{d}\beta}\Big(\frac{1-\beta}{1+\beta}\Big) W_{\beta}'(a-)
+\frac{1-\beta}{1+\beta} \frac{\mathrm{d}W_{\beta}'(a-)}{\mathrm{d}\beta}.
\end{align*}
Plugging \eqref{wa-betadre11} and \eqref{wa-betadre} into the above expression, we have
\begin{align*}
\frac{\mathrm{d}W_{\beta}'(a+)}{\mathrm{d}\beta}
&=\frac{-2}{(1+\beta)^2} \cdot \frac{(1-\beta)\rho_2^--(1+\beta)\rho_1^+}{(1-\beta)(\rho_2^--\rho_1^-)} (\rho_2^-e^{-\rho_1^-a}-\rho_1^-e^{-\rho_2^-a})\\
&\ \ +\frac{1-\beta}{1+\beta} \cdot \frac{-2\rho_1^+}{(1-\beta)^2(\rho_2^--\rho_1^-)}
(\rho_2^-e^{-\rho_1^-a}-\rho_1^-e^{-\rho_2^-a})\\
&=\frac{-2\rho_2^-}{(1+\beta)^2(\rho_2^--\rho_1^-)}
(\rho_2^-e^{-\rho_1^-a}-\rho_1^-e^{-\rho_2^-a}).
\end{align*}
Then $\mathrm{d}W_{\beta}'(a+)/ \mathrm{d}\beta<0$ and  $W_{\beta}'(a+)$ is strictly decreasing in $\beta$ for $\beta\in(-1, 1)$.

For $\beta=0$, by \eqref{wa-} and \eqref{wa221} we obtain
$$W_{0}'(a-)=W_{0}'(a+)=\frac{\rho_2^--\rho_1^+}{\rho_2^--\rho_1^-}
(\rho_2^-e^{-\rho_1^-a}-\rho_1^-e^{-\rho_2^-a}).$$
By the monotonicity of $W_{\beta}'(a-)$ and $W_{\beta}'(a+)$ in $\beta$, we have $W_{\beta}'(a-)<W_{\beta}'(a+)$ if and only if $\beta\in(-1, 0)$.
\end{proof}
\subsection{Proof of Lemma \ref{thm vbc2}}
\begin{proof}\label{thm vbc21}
In the proof we keep $0\leq b_1< a< b_2$ and $0\leq b_1\leq a_1\leq b_2$. For $x\in[0, a_1]$, since no dividend is paid before $X$ reaches $b_1$, and dividends are continuously paid while $X$ remains between $b_1$ and $a_1$ until $X$ decreases to $b_1$, $V_{b_1, a_1, b_2}(x)$ is equivalent to $V_{b_1}(x)$, and by \eqref{val 1}, we have that, for $0\leq x\leq b_1$,
$V_{b_1, a_1, b_2}(x) = \frac{W(x)}{W'(b_1)}$,
and for $b_1< x\leq a_1$,
\begin{align*}
&V_{b_1, a_1, b_2}(x) = x-b_1+\frac{W(b_1)}{W'(b_1)}.
\end{align*}
Particularly, for $x=a_1$, we have
\begin{eqnarray}\label{vbab 256}
V_{b_1, a_1, b_2}(a_1)=a_1-b_1+\frac{W(b_1)}{W'(b_1)}.
\end{eqnarray}

For $a_1< x< b_2$, applying the strong Markov property together with the fact that no dividends are paid out until $X$ exceeds the interval $(a_1, b_2)$, by \eqref{e1}, \eqref{e2} and \eqref{vbab 256} we have
\begin{align}\label{vbab 3}
V_{b_1, a_1, b_2}(x)
&=\mathbb{E}_x[e^{-q \tau_{a_1}}; \tau_{a_1}<\tau_{b_2}]V_{b_1, a_1, b_2}(a_1)
+\mathbb{E}_x[e^{-q \tau_{b_2}}; \tau_{b_2}<\tau_{a_1}]V_{b_1, a_1, b_2}(b_2)\nonumber\\
&=\frac{w(x,b_2)}{w(a_1,b_2)}\big(a_1-b_1+\frac{W(b_1)}{W'(b_1)}\big)
+\frac{w(x,a_1)}{w(b_2,a_1)}V_{b_1, a_1, b_2}(b_2).
	\end{align}
We will now determine the expression for $V_{b_1, a_1, b_2}(b_2)$ by considering two cases: $a\leq a_1$ and $a>a_1$. For $a\leq a_1$, re-applying the strong Markov property, by \eqref{val 6}, \eqref{val 5} and \eqref{vbab 256} we have
\begin{align}\label{bbb2}
V_{b_1, a_1, b_2}(b_2)&=\mathbb{E}_{b_2}\big[\int_0^{\hat{\tau}_{a_1}}e^{-q t} \mathrm{d}D_t^{\pi_{b_1, a_1, b_2}}\big]
+\mathbb{E}_{b_2}[e^{-q \hat{\tau}_{a_1}}] V_{b_1, a_1, b_2}(a_1)\nonumber\\
&=\frac{w(b_2,a_1)}{w_{b_2}(b_2,a_1)}+\frac{w_{b_2}(b_2, b_2)}{w_{b_2}(b_2,a_1)}
\big(a_1-b_1+\frac{W(b_1)}{W'(b_1)}\big).
	\end{align}
For  $a> a_1$, by \eqref{vbab 3} we have
\begin{eqnarray}\label{vbab 1256}
		V_{b_1, a_1, b_2}(a)=\frac{w(a,b_2)}{w(a_1,b_2)}\big(a_1-b_1+\frac{W(b_1)}{W'(b_1)}\big)
+\frac{w(a,a_1)}{w(b_2,a_1)}V_{b_1, a_1, b_2}(b_2),
	\end{eqnarray}
and by \eqref{val 6} and \eqref{val 5}, we have
\begin{align}\label{val uyt}
V_{b_1, a_1, b_2}(b_2)&=\mathbb{E}_{b_2}\big[\int_0^{\hat{\tau}_a}e^{-q t} \mathrm{d}D_t^{\pi_{b_1, a_1, b_2}}\big]
+\mathbb{E}_{b_2}[e^{-q \hat{\tau}_a}] V_{b_1, a_1, b_2}(a)\nonumber\\
&=\frac{w(b_2,a)}{w_{b_2}(b_2,a)}+\frac{w_{b_2}(b_2, b_2)}
{w_{b_2}(b_2,a)}V_{b_1, a_1, b_2}(a),
	\end{align}
and then, by solving a system of equations in \eqref{vbab 1256} and \eqref{val uyt}, we can also find the expression for $V_{b_1, a_1, b_2}(b_2)$ in \eqref{bbb2}. Further plugging \eqref{bbb2} into \eqref{vbab 3}, by $w(x, y)= -w(y, x)$, we get, for $a_1< x< b_2$,
\begin{align*}
 V_{b_1, a_1, b_2}(x)
& = \frac{w(x,a_1)}{w_{b_2}(b_2,a_1)}
+ \Big(\frac{w(x,b_2)}{w(a_1,b_2)} + \frac{w(x,a_1)w_{b_2}(b_2, b_2)}{w(b_2,a_1)w_{b_2}(b_2,a_1)} \Big) \,
\big(a_1-b_1+\frac{W(b_1)}{W'(b_1)}\big)\\
&= \frac{w(x,a_1)}{w_{b_2}(b_2,a_1)}
+ \frac{w(b_2, x)w_{b_2}(b_2,a_1) - w(a_1, x)w_{b_2}(b_2, b_2) }{w(b_2,a_1)w_{b_2}(b_2,a_1)} \,
\big(a_1-b_1+\frac{W(b_1)}{W'(b_1)}\big).
\end{align*}
In the same manner as \eqref{w1w2w3rela}, we have
$$w(b_2, x)w_{b_2}(b_2,a_1) - w(a_1, x)w_{b_2}(b_2, b_2)
= w(b_2,a_1) w_{b_2}(b_2, x),$$
and then,
\begin{align*}
& V_{b_1, a_1, b_2}(x)
= \frac{w(x, a_1)}{w_{b_2}(b_2, a_1)}+\big(a_1-b_1+\frac{W(b_1)}{W'(b_1)}\big)
 \frac{w_{b_2}(b_2, x)}{w_{b_2}(b_2, a_1)}.
\end{align*}

For $b_2< x< \infty$, since no dividends are paid before $X$ reaches $b_2$, by \eqref{bbb2} we have
\begin{align*}
V_{b_1, a_1, b_2}(x)
&=x-b_2+\frac{w(b_2, a_1)}{w_{b_2}(b_2, a_1)}+\big(a_1-b_1+\frac{W(b_1)}{W'(b_1)}\big) \frac{w_{b_2}(b_2, b_2)}{w_{b_2}(b_2, a_1)}.
	\end{align*}
Notice that $V_{b_1, a_1, b_2}\in C(\mathbb{R_+})$.
\end{proof}
\subsection{Proof of Proposition \ref{wb2b2a1u2}}
\begin{proof}\label{wb2b2a1u1}
From $w(x,y)$ defined in \eqref{wxy} and its partial derivative $w_x(x,y)$ in \eqref{wxpart}, together with the associated $g_{1,q}(x)$ and $g_{2,q}(x)$ given in \eqref{g1} and \eqref{g2}, we obtain, for $0< a< a_1\leq b_2$, the form of $w_{b_2}(b_2, a_1)$ is given by \eqref{wb2b1fu3}, which expresses $w_{b_2}(b_2, a_1)$ as the product of three positive terms, implying that $w_{b_2}(b_2, a_1)>0$.

For $0\leq a_1\leq a< b_2$,
\begin{align}
w_{b_2}(b_2, a_1)&=\Big(\rho_2^+\big(1-c_{2}(q)\big) e^{\rho_2^+(b_2-a)}+\rho_1^+c_{2}(q)e^{\rho_1^+(b_2-a)}\Big)\label{wb2b1fu2}\\
&\ \ \cdot \Big(c_{1}(q)e^{\rho_2^-(a_1-a)}+\big(1-c_{1}(q)\big)e^{\rho_1^-(a_1-a)}\Big)
-\rho_1^+ e^{\rho_1^+(b_2-a)}e^{\rho_2^-(a_1-a)}.\nonumber
\end{align}
Separating the terms involving $e^{\rho_2^+(b_2-a)}$ and $e^{\rho_1^+(b_2-a)}$ in \eqref{wb2b1fu2}, we rewrite $w_{b_2}(b_2, a_1)$ for $0\leq a_1\leq a< b_2$ in the form shown in \eqref{wb2b1fu1}. We now prove $w_{b_2}(b_2, a_1)>0$ by considering two cases: $c_2(q)< 0$ and $0\leq c_2(q)< 1$. Note that
\begin{align}\label{certw}
&c_1(q)e^{\rho_2^-(a_1-a)}+\big(1-c_1(q)\big)e^{\rho_1^-(a_1-a)}\nonumber\\
& \geq
c_1(q)e^{\rho_2^-(a_1-a)}+\big(1-c_1(q)\big)e^{\rho_2^- (a_1-a)}
=e^{\rho_2^-(a_1-a)}> 0.
\end{align}
\begin{enumerate}[label=(\roman*)]
\item If $c_2(q)< 0$, then by \eqref{1c1q11} and $e^{\rho_1^-(a_1-a)}\geq1\geq e^{\rho_2^-(a_1-a)}$ we have
$$c_{2}(q)\big(1-c_{1}(q)\big)e^{\rho_1^-(a_1-a)}\leq c_{2}(q)\big(1-c_{1}(q)\big)e^{\rho_2^-(a_1-a)},$$
which leads to
$$\big(1-c_{1}(q)c_{2}(q)\big)e^{\rho_2^-(a_1-a)}
-c_{2}(q)\big(1-c_{1}(q)\big)e^{\rho_1^-(a_1-a)}\geq \big(1-c_{2}(q)\big)e^{\rho_2^-(a_1-a)}> 0.$$
By combining this with \eqref{1c1q11} and \eqref{certw}, we deduce that $w_{b_2}(b_2, a_1)$, as given in \eqref{wb2b1fu1}, is the sum of two positive terms, implying $w_{b_2}(b_2, a_1)>0$.
\item If $0\leq c_2(q)< 1$, then by \eqref{ca66}
\[\rho_2^+\big(1-c_2(q)\big)e^{\rho_2^+(b_2-a)}
+\rho_1^+ c_2(q)e^{\rho_1^+(b_2-a)}>\Big(\rho_2^+\big(1-c_2(q)\big)+ \rho_1^+ c_2(q)\Big)e^{\rho_1^+(b_2-a)}>0. \]
Therefore, $w_{b_2}(b_2, a_1)>0$.
\end{enumerate}
This completes the proof.
\end{proof}
\subsection{Proof of Proposition \ref{va1vb21kl}}
\begin{proof}\label{va1vb21kl34}
We show the differentiability of $V_{b_1, a_1, b_2}(x)$ at $x=a_1$ and $x=b_{2}$. by Lemma \ref{thm vbc2} we have $V_{b_1, a_1, b_2}'(x)=1$ for $x\in[b_1, a_1)$ and then $V_{b_1, a_1, b_2}'(a_1-)=1$. Since  $w_{b_2,b_2}\big(b_{2}, b_{2}\big):=w_{x,y}(x, y)|_{x=y=b_2}=0$, by Lemma \ref{thm vbc2} we obtain
\begin{align*}
&V_{b_1, a_1, b_2}'(b_{2}-)=\frac{w_{b_2}(b_{2}, a_{1})}{w_{b_2}(b_{2}, a_{1})}+
 \frac{w_{b_2,b_2}(b_{2}, b_{2})}{w_{b_2}(b_{2}, a_{1})}V_{b_1, a_1, b_2}(a_{1})
=1,\\
&V_{b_1, a_1, b_2}'(b_{2}+)=1,
\end{align*}
which implies $V_{b_1, a_1, b_2}'(b_{2})=1$.
\end{proof}
\subsection{Proof of Lemma \ref{lem b1a125}}
\begin{proof}\label{b1a1225}
In the proof we keep $0\leq b_1\leq a_1 < a< b_2$. Preliminarily evaluating $V_{b_1, a_1, b_2}'(x)$ and $V_{b_1, a_1, b_2}''(x)$, as defined in \eqref{vb1a1b266} and \eqref{vsecde}, respectively, at $x=a_1$ yields
\begin{align}
&V_{b_1, a_1, b_2}'(a_1)
=\frac{\tilde{K}_{2}(b_1, a_1, b_2){\rho_2^-} e^{\rho_2^-(a_1-a)}-\tilde{K}_{1}(b_1, a_1, b_2){\rho_1^-} e^{\rho_1^-(a_1-a)}}{w_{b_2}(b_2, a_1)},\nonumber\\
&V_{b_1, a_1, b_2}''(a_1)
=\frac{\tilde{K}_{2}(b_1, a_1, b_2){\rho_2^-}^2 e^{\rho_2^-(a_1-a)}-\tilde{K}_{1}(b_1, a_1, b_2){\rho_1^-}^2 e^{\rho_1^-(a_1-a)}}{w_{b_2}(b_2, a_1)}.\label{vsecghj}
\end{align}
By referring to the definitions of $g_{i,q}(x)$ for $i=1, 2$ in \eqref{g1} and \eqref{g2}, $\tilde{K}_{1}(b_1, a_1, b_2)$ and $\tilde{K}_{2}(b_1, a_1, b_2)$, as given in \eqref{k1b1a1} and \eqref{k1b1a2}, respectively, can be expanded as
\begin{align}
&\tilde{K}_{1}(b_1, a_1, b_2)
=\big(1-c_1(q)\big)\label{kkk111}\\
& \ \ \ \ \ \ \ \ \  \cdot
\Big(e^{\rho_2^- (a_1-a)}-V_{b_1, a_1, b_2}(a_1)\big(\rho_2^+(1-c_2(q)) e^{\rho_2^+(b_2-a)}+\rho_1^+c_2(q) e^{\rho_1^+(b_2-a)}\big)\Big),\nonumber\\
&\tilde{K}_{2}(b_1, a_1, b_2)
=\big(1-c_1(q)\big)e^{\rho_1^- (a_1-a)}\label{kkk222}\\
&\ \ \ \ \ \ \ \ \
+V_{b_1, a_1, b_2}(a_1)\Big(\rho_2^+ c_1(q)\big(1-c_2(q)\big) e^{\rho_2^+ (b_2-a)}
-\rho_1^+ \big(1-c_1(q) c_2(q)\big) e^{\rho_1^+(b_2-a)}\Big).\nonumber
\end{align}
Then, substituting \eqref{kkk111} and \eqref{kkk222} into  $V_{b_1, a_1, b_2}'(a_1)$,
we obtain
\begin{align*}
&V_{b_1, a_1, b_2}'(a_1) w_{b_2}(b_2, a_1)
=
({\rho_2^-}-{\rho_1^-})\big(1-c_1(q)\big)e^{(\rho_2^- +\rho_1^-)(a_1-a)}\\
& +V_{b_1, a_1, b_2}(a_1) \Big(
\rho_2^+ \big(1-c_2(q)\big)\big({\rho_2^-} c_1(q) e^{\rho_2^- (a_1-a)}
+{\rho_1^-} (1-c_1(q)) e^{\rho_1^- (a_1-a)}\big)e^{\rho_2^+(b_2-a)}\nonumber\\
&\ \ \ \ \ \ \ \ \ \ \ \ \ \ -\rho_1^+\big({\rho_2^-} (1-c_1(q)c_2(q)) e^{\rho_2^- (a_1-a)}
-{\rho_1^-} c_2(q) (1-c_1(q)) e^{\rho_1^- (a_1-a)}\big)e^{\rho_1^+(b_2-a)}\Big).\nonumber
\end{align*}
Since $V_{b_1,a_1, b_2}'(a_1)=1$, by \eqref{wb2b1fu1} we have
\begin{align}\label{asdw11}
&({\rho_2^-}-{\rho_1^-})\big(1-c_1(q)\big)e^{(\rho_2^- +\rho_1^-)(a_1-a)}\nonumber\\
& =\rho_2^+ \big(1-c_{2}(q)\big)\Big(c_{1}(q)e^{\rho_2^-(a_1-a)}+\big(1-c_{1}(q)\big)e^{\rho_1^-(a_1-a)}\Big)
e^{\rho_2^+(b_2-a)}\nonumber\\
&\ \ -\rho_1^+\Big(\big(1-c_{1}(q)c_{2}(q)\big)e^{\rho_2^-(a_1-a)}
-c_{2}(q)\big(1-c_{1}(q)\big)e^{\rho_1^-(a_1-a)}\Big) e^{\rho_1^+(b_2-a)}\nonumber\\
& \ \ -V_{b_1, a_1, b_2}(a_1)
\Big(
\rho_2^+ \big(1-c_2(q)\big)\big({\rho_2^-} c_1(q) e^{\rho_2^- (a_1-a)}
+{\rho_1^-} (1-c_1(q)) e^{\rho_1^- (a_1-a)}\big)e^{\rho_2^+(b_2-a)}\nonumber\\
&\ \ \ \ \ \ \ \ \ \ \ \ \ \ \ \ \ \ \ \ \ \ \ \ -\rho_1^+\big({\rho_2^-} (1-c_1(q)c_2(q)) e^{\rho_2^- (a_1-a)}
-{\rho_1^-} c_2(q) (1-c_1(q)) e^{\rho_1^- (a_1-a)}\big)e^{\rho_1^+(b_2-a)}\Big)\nonumber\\
&=\rho_2^+\big(1-c_2(q)\big)\Big(c_1(q)\big(1-\rho_2^-V_{b_1,a_1, b_2}(a_1)\big)e^{\rho_2^- (a_1-a)}\nonumber\\
&\ \ \ \ \ \ \ \ \ \ \ \ \ \ \ \ \ \ \ \ \ \ \ \ \ \ +\big(1-c_1(q)\big)\big(1-\rho_1^-V_{b_1,a_1, b_2}(a_1)\big)e^{\rho_1^- (a_1-a)}\Big)e^{\rho_2^+ (b_2-a)}\nonumber\\
&\ \ -\rho_1^+\Big(\big(1-c_1(q)c_2(q)\big)\big(1-\rho_2^-V_{b_1,a_1, b_2}(a_1)\big)e^{\rho_2^- (a_1-a)}\nonumber\\
&\ \ \ \ \ \ \ \ \ \ \ -c_2(q)\big(1-c_1(q)\big)\big(1-\rho_1^-V_{b_1,a_1, b_2}(a_1)\big)e^{\rho_1^- (a_1-a)}\Big)e^{\rho_1^+ (b_2-a)}.
\end{align}
For $\tilde{K}_{1}(b_1, a_1, b_2)$ as given in \eqref{kkk111}, by \eqref{asdw11} we have
\begin{align*}
&\tilde{K}_{1}(b_1, a_1, b_2)\\
&=\frac{e^{-\rho_1^- (a_1-a)}}{\rho_2^- -\rho_1^-}
 \Big(\rho_2^+\big(1-c_2(q)\big)\big(c_1(q)\big(1-\rho_2^-V_{b_1,a_1, b_2}(a_1)\big)e^{\rho_2^- (a_1-a)}\nonumber\\
&\ \ \ \ \ \ \ \ \ \ \ \ \ \ \ \ \ \ \ \ \ \ \ \ \ \ \ \ \ \ \ \ \ \ \ \ \ \ \ \ \ \ \ \ \ \ +\big(1-c_1(q)\big)\big(1-\rho_1^-V_{b_1,a_1, b_2}(a_1)\big)e^{\rho_1^- (a_1-a)}\big)e^{\rho_2^+ (b_2-a)}\nonumber\\
&\ \ \ \ \ \ \ \ \ \ \ \ \ \ \ \ \ \ \ \ \ \ -\rho_1^+\big(\big(1-c_1(q)c_2(q)\big)\big(1-\rho_2^-V_{b_1,a_1, b_2}(a_1)\big)e^{\rho_2^- (a_1-a)}\nonumber\\
&\ \ \ \ \ \ \ \ \ \ \ \ \ \ \ \ \ \ \ \ \ \ \ \ \ \ \ \ \ \ \ -c_2(q)\big(1-c_1(q)\big)\big(1-\rho_1^-V_{b_1,a_1, b_2}(a_1)\big)e^{\rho_1^- (a_1-a)}\big)e^{\rho_1^+ (b_2-a)}\Big)\nonumber\\
& \ \ - \big(1-c_1(q)\big) V_{b_1, a_1, b_2}(a_1)\Big(\rho_2^+\big(1-c_2(q)\big) e^{\rho_2^+(b_2-a)}+\rho_1^+c_2(q) e^{\rho_1^+(b_2-a)}\Big).
\end{align*}
Grouping the terms involving $V_{b_1, a_1, b_2}(a_1)$, the above expression becomes
\begin{align*}
&\tilde{K}_{1}(b_1, a_1, b_2)\nonumber\\
&=\frac{1}{\rho_2^- -\rho_1^-}
\Big(\rho_2^+\big(1-c_2(q)\big)\big(c_1(q) e^{(\rho_2^- - \rho_1^-) (a_1-a)}+\big(1-c_1(q)\big)\big)e^{\rho_2^+ (b_2-a)}\nonumber\\
&  \ \ \ \ \ \ \ \ \ \ \ \ \ \ \ \ \ \ \ -\rho_1^+\big(\big(1-c_1(q)c_2(q)\big) e^{(\rho_2^- -\rho_1^-) (a_1-a)} -c_2(q)\big(1-c_1(q)\big)\big)e^{\rho_1^+ (b_2-a)}\Big)\nonumber\\
& \ \
- \frac{V_{b_1,a_1, b_2}(a_1)}{\rho_2^- -\rho_1^-} \Big(\rho_2^+\big(1-c_2(q)\big)\big(c_1(q)\rho_2^- e^{(\rho_2^- - \rho_1^-) (a_1-a)}
+\big(1-c_1(q)\big)\rho_1^-\big)e^{\rho_2^+ (b_2-a)}\nonumber\\
& \ \ \ \ \ \ \ \ \ \ \ \ \ \ \ \ \ \ \ \ \ \ \ \ \ \ -\rho_1^+\big(\big(1-c_1(q)c_2(q)\big)\rho_2^- e^{(\rho_2^- -\rho_1^-) (a_1-a)} -c_2(q)\big(1-c_1(q)\big) \rho_1^-\big)e^{\rho_1^+ (b_2-a)}\Big)\nonumber\\
& \ \ - V_{b_1, a_1, b_2}(a_1) \big(1-c_1(q)\big) \Big(\rho_2^+\big(1-c_2(q)\big) e^{\rho_2^+(b_2-a)}+\rho_1^+c_2(q) e^{\rho_1^+(b_2-a)}\Big)\nonumber\\
&=\frac{1-\rho_2^-V_{b_1, a_1, b_2}(a_1)}{\rho_2^- -\rho_1^- }
\Big(\rho_2^+\big(1-c_2(q)\big)\big(c_1(q) e^{(\rho_2^--\rho_1^-)(a_1-a)}+(1-c_1(q))\big)
e^{\rho_2^+ (b_2-a)}\nonumber\\
&\ \ \ \ \ \ \ \ \ \ \ \ \ \ \ \ \ \ \ \  \ \ \ \ \ \ \ \ \ \ \ \ \ \ \ \ \ \ \ -\rho_1^+\big((1-c_1(q)c_2(q)) e^{(\rho_2^--\rho_1^-)(a_1-a)}-c_2(q) (1-c_1(q))\big)e^{\rho_1^+ (b_2-a)}\Big)\nonumber\\
&=\frac{1-\rho_2^-V_{b_1, a_1, b_2}(a_1)}{\rho_2^- -\rho_1^-}
w_{b_2}(b_2, a_1) e^{-\rho_1^-(a_1-a)}.\nonumber
\end{align*}
Similarly, by \eqref{kkk222} and \eqref{asdw11},
\begin{align*}
&\tilde{K}_{2}(b_1, a_1, b_2)\\
&=\frac{e^{-\rho_2^- (a_1-a)}}{\rho_2^- -\rho_1^-}
  \Big(\rho_2^+\big(1-c_2(q)\big)\big(c_1(q)\big(1-\rho_2^-V_{b_1,a_1, b_2}(a_1)\big)e^{\rho_2^- (a_1-a)}\nonumber\\
&\ \ \ \ \ \ \ \ \ \ \ \ \ \ \ \ \ \ \ \ \ \ \ \ \ \ \ \ \ \ \ \ \ \ \ \ \ \ \ \ \ \ \ \ \ \ \ +\big(1-c_1(q)\big)\big(1-\rho_1^-V_{b_1,a_1, b_2}(a_1)\big)e^{\rho_1^- (a_1-a)}\big)e^{\rho_2^+ (b_2-a)}\nonumber\\
&\ \ \ \ \ \ \ \ \ \ \ \ \ \ \ \ \ \ \ \ \ -\rho_1^+\big(\big(1-c_1(q)c_2(q)\big)\big(1-\rho_2^-V_{b_1,a_1, b_2}(a_1)\big)e^{\rho_2^- (a_1-a)}\nonumber\\
&\ \ \ \ \ \ \ \ \ \ \ \ \ \ \ \ \ \ \ \ \ \ \ \ \ \ \ \ \ -c_2(q)\big(1-c_1(q)\big)\big(1-\rho_1^-V_{b_1,a_1, b_2}(a_1)\big)e^{\rho_1^- (a_1-a)}\big)e^{\rho_1^+ (b_2-a)}\Big)\nonumber\\
&\ \
+V_{b_1, a_1, b_2}(a_1)\Big(\rho_2^+ c_1(q)\big(1-c_2(q)\big) e^{\rho_2^+ (b_2-a)}
-\rho_1^+ \big(1-c_1(q) c_2(q)\big) e^{\rho_1^+(b_2-a)}\Big)\nonumber\\
&=\frac{1}{\rho_2^- -\rho_1^-}
\Big(\rho_2^+\big(1-c_2(q)\big)\big(c_1(q) +\big(1-c_1(q)\big)e^{(\rho_1^- - \rho_2^-) (a_1-a)}\big)e^{\rho_2^+ (b_2-a)}\nonumber\\
&\ \ \ \ \ \ \ \ \ \ \ \ \ \ \ \ \ \ \ -\rho_1^+\big(\big(1-c_1(q)c_2(q)\big) -c_2(q)\big(1-c_1(q)\big)e^{(\rho_1^- -\rho_2^-) (a_1-a)} \big)e^{\rho_1^+ (b_2-a)}\Big)\nonumber\\
&\ \
- \frac{V_{b_1,a_1, b_2}(a_1)}{\rho_2^- -\rho_1^-} \Big(\rho_2^+\big(1-c_2(q)\big)\big(c_1(q)\rho_2^-
+\big(1-c_1(q)\big)\rho_1^-e^{(\rho_1^- - \rho_2^-) (a_1-a)}\big)e^{\rho_2^+ (b_2-a)}\nonumber\\
& \ \ \ \ \ \ \ \ \ \ \ \ \ \ \ \ \ \ \ \ \ \ \ \ \ \   -\rho_1^+\big(\big(1-c_1(q)c_2(q)\big)\rho_2^-  -c_2(q)\big(1-c_1(q)\big) \rho_1^- e^{(\rho_1^- -\rho_2^-) (a_1-a)} \big)e^{\rho_1^+ (b_2-a)}\Big)\nonumber\\
& \ \ +V_{b_1, a_1, b_2}(a_1)\Big(\rho_2^+ c_1(q)\big(1-c_2(q)\big) e^{\rho_2^+ (b_2-a)}
-\rho_1^+ \big(1-c_1(q) c_2(q)\big) e^{\rho_1^+(b_2-a)}\Big)\\
&=\frac{1-\rho_1^-V_{b_1, a_1, b_2}(a_1)}{\rho_2^- -\rho_1^- }
\Big(\rho_2^+\big(1-c_2(q)\big)\big(c_1(q) +(1-c_1(q))e^{(\rho_1^--\rho_2^-)(a_1-a)}\big)
e^{\rho_2^+ (b_2-a)}
\end{align*}
\begin{align*}
&\ \ \ \ \ \ \ \ \ \ \ \ \ \ \ \ \ \ \ \ \ \ \ \ \ \ \ \ \  -\rho_1^+\big((1-c_1(q)c_2(q)) -c_2(q) (1-c_1(q))e^{(\rho_1^--\rho_2^-)(a_1-a)}\big)e^{\rho_1^+ (b_2-a)}\Big)\nonumber\\
&=\frac{1-\rho_1^-V_{b_1, a_1, b_2}(a_1)}{\rho_2^- -\rho_1^-}
w_{b_2}(b_2, a_1) e^{-\rho_2^-(a_1-a)}.
\end{align*}

Next, we analyze the expression for $V_{b_1, a_1, b_2}''(a_1)$, under the condition that $V_{b_1,a_1, b_2}'(a_1)=1$ with $a_1\in[b_1, a)$. By substituting $\tilde{K}_{1}(b_1, a_1, b_2)$ with \eqref{fde33216} and $\tilde{K}_{2}(b_1, a_1, b_2)$ with \eqref{fde3311} in \eqref{vsecghj}, $V_{b_1, a_1, b_2}''(a_1)$ can be reformulated as
\begin{align*}
V_{b_1, a_1, b_2}''(a_1)
&=\frac{1-\rho_1^-V_{b_1, a_1, b_2}(a_1)}{\rho_2^- -\rho_1^-}
w_{b_2}(b_2, a_1) e^{-\rho_2^-(a_1-a)}\frac{{\rho_2^-}^2 e^{\rho_2^-(a_1-a)}}{w_{b_2}(b_2, a_1)}\\
&\ \  -\frac{1-\rho_2^-V_{b_1, a_1, b_2}(a_1)}{\rho_2^- -\rho_1^-}
w_{b_2}(b_2, a_1) e^{-\rho_1^-(a_1-a)} \frac{{\rho_1^-}^2 e^{\rho_1^-(a_1-a)}}{w_{b_2}(b_2, a_1)}\\
&=\frac{{\rho_2^-}^2\big(1-\rho_1^-V_{b_1, a_1, b_2}(a_1)\big)-{\rho_1^-}^2\big(1-\rho_2^-V_{b_1, a_1, b_2}(a_1)\big)}{\rho_2^- -\rho_1^-}\\
& = \rho_2^-+\rho_1^--\rho_1^- \rho_2^- V_{b_1, a_1, b_2}(a_1).
\end{align*}
Then \eqref{v2b12a1er} follows from $\rho_2^-+\rho_1^-= -2\mu_-$ and $\rho_1^- \rho_2^-=-2q$.
\end{proof}
\subsection{Proof of Lemma \ref{lem b1a1}}
\begin{proof}\label{b1a12}
In the proof we keep $0\leq b_1\leq a_1 < a< b_2$. If $V_{b_1,a_1, b_2}'(a_1)=1$ with $a_1\in[b_1, a)$, by \eqref{v2b12a1er} we have  $V_{b_1, a_1, b_2}''(a_1)> 0$ if and only if $V_{b_1, a_1, b_2}(a_1)> {\mu_-}/{q}$.

We first prove that $V_{b_1, a_1, b_2}''(a_1)\geq 0$ by verifying the above condition.
For $b_1=b_-\in(0,a_1)$, it follows from
  \eqref{val 1} and Lemma \ref{thm vbc2}  that $V_{b_-, a_1, b_2}(b_-)=V_{b_-}(b_-)={\mu_-}/{q}$ by \eqref{mulg}, and then
\begin{align*}
&V_{b_-, a_1, b_2}(a_1)> V_{b_-, a_1, b_2}(b_-)={\mu_-}/{q},
\end{align*}
which implies $V_{b_-, a_1, b_2}''(a_1)> 0$; for $b_1=b_-=a_1 \in(0, a)$, since $V_{b_-, b_-, b_2}(a_1)=V_{b_-, b_-, b_2}(b_-)={\mu_-}/{q}$, by \eqref{v2b12a1er} we have $V_{b_-, b_-, b_2}''(a_1)=0$; when $b_1=0$ for $b_-\leq0$, by Proposition \ref{b1min} (i) we have $\mu_-\leq0$, and then  we have  $V_{0, a_1, b_2}''(a_1)>0$ for $a_1\in(0, a)$ by \eqref{v2b12a1er}, and $V_{0, 0, b_2}''(a_1)=-2\mu_-\geq 0$ for
$a_1=0$.

We next show the monotonicity of $V_{b_1, a_1, b_2}'(x)$ for $x\in(a_1, a)$
by showing $V_{b_1, a_1, b_2}''(x)>0$.
Note that $w_{b_2}(b_2, a_1)>0$ by Proposition \ref{wb2b2a1u2}.
Under the condition $V_{b_1,a_1, b_2}'(a_1)=1$ for $a_1\in[b_1, a)$, by Lemma \ref{lem b1a125} we have that $\tilde{K}_{1}(b_1, a_1, b_2)$ can be either non-positive or positive, and
\begin{align}\label{fde331}
&\tilde{K}_{2}(b_1, a_1, b_2)>0.
\end{align}
If $\tilde{K}_{1}(b_1, a_1, b_2)\leq0$, then by \eqref{vsecde} and \eqref{fde331}, $V_{b_1, a_1, b_2}''(x)>0$ since it is the sum of two positive terms.
If $\tilde{K}_{1}(b_1, a_1, b_2)>0$,
observing that by \eqref{kkk111} and \eqref{kkk222},
\begin{align*}
&\tilde{K}_{2}(b_1, a_1, b_2)-\tilde{K}_{1}(b_1, a_1, b_2)\nonumber\\
&=\big(1-c_1(q)\big)e^{\rho_1^- (a_1-a)}- \big(1-c_1(q)\big) e^{\rho_2^- (a_1-a)}\nonumber\\
&\ \
+V_{b_1, a_1, b_2}(a_1)\Big(\rho_2^+ c_1(q)\big(1-c_2(q)\big) e^{\rho_2^+ (b_2-a)}
-\rho_1^+ \big(1-c_1(q) c_2(q)\big) e^{\rho_1^+(b_2-a)}\Big)\nonumber\\
& \ \ + V_{b_1, a_1, b_2}(a_1) \big(1-c_1(q)\big) \Big(\rho_2^+\big(1-c_2(q)\big) e^{\rho_2^+(b_2-a)}+\rho_1^+c_2(q) e^{\rho_1^+(b_2-a)}\Big)\nonumber\\
&=\big(1-c_{1}(q)\big)
\big(e^{\rho_1^-(a_1-a)}-e^{\rho_2^-(a_1-a)}\big)\nonumber\\
& \ \ +V_{b_1, a_1, b_2}(a_1)\big(1-c_{2}(q)\big)
\big(\rho_2^+e^{\rho_2^+(b_2-a)}-\rho_1^+e^{\rho_1^+(b_2-a)}\big)>0,
\end{align*}
then by \eqref{fde331} we have $\tilde{K}_{2}(b_1, a_1, b_2)>\tilde{K}_{1}(b_1, a_1, b_2)>0$ and
\begin{align*}
V_{b_1, a_1, b_2}'''(x)
&=\frac{\tilde{K}_{2}(b_1, a_1, b_2){\rho_2^-}^3e^{\rho_2^-(x-a)}-\tilde{K}_{1}(b_1, a_1, b_2){\rho_1^-}^3e^{\rho_1^-(x-a)}}{w_{b_2}(b_2, a_1)}\\
&>\frac{\tilde{K}_{1}(b_1, a_1, b_2)\big({\rho_2^-}^3e^{\rho_2^-(x-a)}-{\rho_1^-}^3e^{\rho_1^-(x-a)}\big)}{w_{b_2}(b_2, a_1)}>0,
\end{align*}
which implies that $V_{b_1, a_1, b_2}''(x)$ is strictly increasing in $x$ and satisfies $V_{b_1, a_1, b_2}''(x)> V_{b_1, a_1, b_2}''(a_1)\geq 0$.
\end{proof}
\subsection{Proof of Lemma \ref{lem k12122}}
\begin{proof}\label{lem k12222}
In the proof we assume that $0\leq b_1 < a < b_2$ and $0\leq b_1\leq a_1 < b_2$ and take $x\in[a_1, b_{2})\cap (a, b_2)$ unless stated otherwise, while using the fact that $w_{b_2}(b_2, a_1)>0$ as shown in Proposition \ref{wb2b2a1u2}.
By \eqref{vb1a1b222},
\begin{align}\label{vsecb2re}
& V_{b_1,a_1, b_2}''(b_2)=\frac{\hat{K}_{2}(b_1, a_1, b_2){\rho_2^+}^2e^{\rho_2^+(b_2-a)}-\hat{K}_{1}(b_1, a_1, b_2){\rho_1^+}^2e^{\rho_1^+(b_2-a)}}{w_{b_2}(b_2, a_1)}.
\end{align}
By \eqref{f11}, for $a_1<a$,
\begin{align*}
&\hat{K}_{2}(b_1, a_1, b_2)\\
&=
\big(1-c_{2}(q)\big)\Big(c_{1}(q)e^{\rho_2^-(a_1-a)}
+\big(1-c_{1}(q)\big)e^{\rho_1^-(a_1-a)}
-\rho_1^+e^{\rho_1^+(b_2-a)}V_{b_1, a_1, b_2}(a_1)\Big)\nonumber\\
&>0,\nonumber
\end{align*}
applying \eqref{certw}, and for $a_1\geq a$,
\begin{align*}
\hat{K}_{2}(b_1, a_1, b_2)&=\big(1-c_{2}(q)\big)\big(e^{\rho_1^+(a_1-a)}
-\rho_1^+e^{\rho_1^+(b_2-a)}V_{b_1, a_1, b_2}(a_1)\big)>0,
\end{align*}
since it is a product of two positive terms.

If there exists $b_2>(a_1 \vee a)$ such that $V_{b_1,a_1, b_2}''(b_2)=0$, then by \eqref{vsecb2re} we have
\begin{eqnarray*}\label{3b3}
&&b_{2}:=a+\frac{1}{\rho_2^+-\rho_1^+}\ln \hat{K}(b_1, a_{1}, b_{2}),
\end{eqnarray*}
where $\hat{K}(b_1, a_{1}, b_{2}):=
{\rho_1^+}^2\hat{K}_{1}(b_1, a_{1}, b_{2})/
\big({\rho_2^+}^2\hat{K}_{2}(b_1, a_{1}, b_{2})\big)$.
For $b_{2}>a$ we have $\hat{K}(b_1, a_{1}, b_{2})>1$, which combined with $\hat{K}_{2}(b_1, a_{1}, b_{2})>0$ leads to
\begin{align*}
&\hat{K}_{1}(b_1, a_{1}, b_{2}) >\frac{{\rho_2^+}^2}{{\rho_1^+}^2} \hat{K}_{2}(b_1, a_{1}, b_{2})>0.
\end{align*}

Next, we prove that $V_{b_1, a_1, b_2}'(x)$ is strictly decreasing  under the condition $V_{b_1,a_1, b_2}''(b_2)=0$ for $b_{2}>a$.
By \eqref{vb1a1b222},
\begin{align*}
V_{b_1, a_1, b_2}'''(x)&=
\frac{\hat{K}_{2}(b_1, a_1, b_2){\rho_2^+}^3e^{\rho_2^+(x-a)}-\hat{K}_{1}(b_1, a_1, b_2){\rho_1^+}^3e^{\rho_1^+(x-a)}}{w_{b_2}(b_2, a_1)}>0.
\end{align*}
Then $V_{b_1, a_1, b_2}''(x)$ is strictly increasing in $x$ and $V_{b_1, a_1, b_2}''(x)< V_{b_1, a_1, b_2}''(b_{2})=0$. Then $V_{b_1, a_1, b_2}'(x)$ is strictly decreasing and $V_{b_1, a_1, b_2}'(x)>V_{b_1, a_1, b_2}'(b_{2})$.
\end{proof}

\subsection{Proof of Remark \ref{orderv1v2}}
\begin{proof}\label{orderv1v200}
(i) By \eqref{val 1} and Lemma \ref{thm vbc2}, the comparison between barrier and band value functions on $[0,b_1)$ is determined by the location of $\arg\min W'$.
If the minimum of $W'$ is attained only at $a+$ or $b_+$, the corresponding barrier value strictly exceeds the band value on $[0,b_1)$, hence no $(b_1,a_1, b_2)$-band is optimal. Conversely, a necessary condition for band optimality is that $W'$ attain its minimum at one of $0+, b_-, a-$.

(ii) Since $\mathcal M\in\{0+, b_-, a-\}$, it follows directly that $V_{b_1}(x)>\max\{V_{a+}(x),V_{b_+}(x)\}$.
Assume $V_{b_1,a_1, b_2}'(a_1)=1$ for some $a_1\in[b_1, a)$ and consider $x\in(a_1, a)$. By Lemma \ref{lem b1a1} we have $V_{b_1, a_1, b_2}'(x)$ is strictly increasing, and then, $V_{b_1, a_1, b_2}'(x)> V_{b_1,a_1, b_2}'(a_1)=1$.
From \eqref{val 1}, we also know that $V_{b_1}'(x)=1$ for all $x\in(a_1, a)$. Consequently, $V_{b_1, a_1, b_2}'(x)- V_{b_1}'(x)>0$. Integrating over $(a_1, x)$ and using $V_{b_1, a_1, b_2}(a_1)=V_{b_1}(a_1)$ yields $V_{b_1, a_1, b_2}(x)> V_{b_1}(x)$ for every $x\in(a_1, a)$.

(iii) We omit the proof of (iii), which follows from Lemma \ref{lem k12122} together with $V_{b_1,a_1, b_2}'(b_2)=1$ (proved in Proposition \ref{va1vb21kl}) by the same argument as in (ii).
\end{proof}

\subsection{Proof of Lemma \ref{thm vbb}}
\begin{proof}\label{thm vbbproof}
For $a \leq a_1$,  $V_{b_1, a_1, b_2}(x)$ is a special case of Lemma \ref{thm vbc2}, and we focus on the case $x \geq a_1$. For $x\in[a_1, b_2)$, by \eqref{wb2b1fu3} and
\begin{align*}
&w(x, a_1)= \big(1-c_{2}(q)\big)e^{(\rho_2^++\rho_1^+)(a_1-a)}
\big(e^{\rho_2^+(x-a_1)}- e^{\rho_1^+(x-a_1)}\big),\\
&w_{b_2}(b_2, x)= \big(1-c_{2}(q)\big)e^{(\rho_2^++\rho_1^+)(x-a)}
\big(\rho_2^+e^{\rho_2^+(b_2-x)}-\rho_1^+ e^{\rho_1^+(b_2-x)}\big),
\end{align*}
 applying $W_+(x)$ given in \eqref{w+xeft}, we have
\begin{align*}
&\frac{w(x, a_1)}{w_{b_2}(b_2, a_1)}=\frac{e^{\rho_2^+(x-a_1)}- e^{\rho_1^+(x-a_1)} }
{\rho_2^+ e^{\rho_2^+(b_2-a_1)}-\rho_1^+ e^{\rho_1^+(b_2-a_1)}}=\frac{W_+(x-a_1)}{W_+'(b_2-a_1)},\\
&\frac{w_{b_2}(b_2, x)}{w_{b_2}(b_2, a_1)}
=\frac{e^{(\rho_2^++\rho_1^+)(x-a)}
\big(\rho_2^+e^{\rho_2^+(b_2-x)}-\rho_1^+ e^{\rho_1^+(b_2-x)}\big)}
{e^{(\rho_2^++\rho_1^+)(a_1-a)}
\big(\rho_2^+ e^{\rho_2^+(b_2-a_1)}-\rho_1^+ e^{\rho_1^+(b_2-a_1)}\big)}=\frac{W_+'(b_2-x)e^{(\rho_2^++\rho_1^+)(x-a_1)}}{W_+'(b_2-a_1)}.
\end{align*}
In particular, at $x=b_2$, we obtain the following result
\begin{align*}
&\frac{w(b_2, a_1)}{w_{b_2}(b_2, a_1)}=\frac{W_+(b_2-a_1)}{W_+'(b_2-a_1)} \ \ \ \mathrm{and} \ \
\frac{w_{b_2}(b_2, b_2 )}{w_{b_2}(b_2, a_1)}
=\frac{W_+'(0)e^{(\rho_2^++\rho_1^+)(b_2-a_1)}}{W_+'(b_2-a_1)}.
\end{align*}
Substituting these results into Lemma \ref{thm vbc2} yields Lemma \ref{thm vbb}.
\end{proof}
\subsection{Proof of Lemma \ref{lem s11}}
\begin{proof}\label{s11}
Fix $0\leq b_1<a<b_2$. By Proposition \ref{va1vb21kl}, we have $V_{b_1, a, b_2}'(b_2)=1$. If $V_{b_1,a, b_2}''(b_2)=0$, then by Lemma \ref{lem k12122} we have $V_{b_1,a_1, b_2}'(x)$ is strictly decreasing for $x\in (a, b_2)$, and then $V_{b_1, a, b_2}'(a+)> V_{b_1,a, b_2}'(b_2)=1$. It follows that $S'(\beta)= V_{b_1, a, b_2}'(a+)+1> 2$, and then, $S(\beta)$ is strictly increasing for $\beta\in(-1,1)$.

Since $\lim_{\beta\downarrow -1}S(\beta)=-2<0$ and $S(0)=V_{b_1, a, b_2}'(a+)-1>0$, there exists a unique $\beta^*\in(-1,0)$,  given by \eqref{al2}, satisfying $S(\beta^*)=0$ and $S(\beta)\leq 0$ for $\beta\in(-1, \beta^*]$.
\end{proof}

\subsection{Proof of Remark \ref{prop:decision}}
\begin{proof}\label{prop:decision00}
For $\mathcal{M}\in\{a+,\,b_+\}$, the conclusion follows directly from Theorems \ref{thm va+}-\ref{thm vb2}.
Throughout $(C1)$-$(C3)$ we assume $\mathcal M\in\{0+,\,b_-,\,a-\}$. The justifications are:
$(C1)\text{-(i)}$ from Theorem \ref{thm vbb5} together with Remark \ref{orderv1v2}-(ii); \((C1)\text{-(ii)}\) from Theorems \ref{thm v0+}-\ref{thm va-};
\((C2)\text{-(i)}\) with \(V''_{b_1,a_1,b_2}(b_2)=0\) from Theorems \ref{thm vabb51} and Remark \ref{orderv1v2}-(iii);
\((C2)\text{-(i)}\) with \(V''_{b_1,a_1,b_2}(b_2)\neq 0\) and \((C2)\text{-(ii)}\) from Theorems \ref{thm v0+}-\ref{thm va-};
and \((C3)\) from Remark \ref{rem52}.
\end{proof}

\end{appendices}


\bibliographystyle{plain}
\bibliography{bibliography}


\end{document}